\documentclass[11pt,reqno]{amsart}
\usepackage[a4paper, hmargin={2.7cm,2.7cm},vmargin={3.3cm,3.3cm}]{geometry}
\usepackage[english]{babel}
\usepackage{color}
\usepackage{amsmath}
\usepackage{amssymb}
\usepackage{amsfonts}
\usepackage{amsthm}
\usepackage{enumerate}
\usepackage{amsthm}
\usepackage[noadjust]{cite}
\usepackage{dsfont}
\usepackage{etoolbox}
\usepackage[unicode=true]{hyperref}
\usepackage[noabbrev]{cleveref}

\usepackage{todonotes}

\makeatletter
\patchcmd{\ttlh@hang}{\parindent\z@}{\parindent\z@\leavevmode}{}{}
\patchcmd{\ttlh@hang}{\noindent}{}{}{}
\makeatother

\newcommand\numberthis{\addtocounter{equation}{1}\tag{\theequation}}

\makeatletter
\@namedef{subjclassname@2020}{\textup{2020} Mathematics Subject Classification}
\makeatother

\theoremstyle{plain}
\newtheorem{theorem}{Theorem}[section]
\newtheorem{lemma}[theorem]{Lemma}
\newtheorem{proposition}[theorem]{Proposition}

\newtheorem{innercustomgeneric}{\customgenericname}
\providecommand{\customgenericname}{}
\newcommand{\newcustomtheorem}[2]{%
  \newenvironment{#1}[1]
  {%
   \renewcommand\customgenericname{#2}%
   \renewcommand\theinnercustomgeneric{##1}%
   \innercustomgeneric
  }
  {\endinnercustomgeneric}
}

\newenvironment{dcases} {\left\lbrace
    \begin{aligned}}
    {\end{aligned}\right\rbrace
}

\def\Xint#1{\mathchoice
{\XXint\displaystyle\textstyle{#1}}%
{\XXint\textstyle\scriptstyle{#1}}%
{\XXint\scriptstyle\scriptscriptstyle{#1}}%
{\XXint\scriptscriptstyle\scriptscriptstyle{#1}}%
\!\int}
\def\XXint#1#2#3{{\setbox0=\hbox{$#1{#2#3}{\int}$ }
\vcenter{\hbox{$#2#3$ }}\kern-.6\wd0}}

\def\dashint{\Xint-}

\newcustomtheorem{customthm}{Theorem}
\newcustomtheorem{customlemma}{Lemma}

\theoremstyle{definition}
\newtheorem{definition}[theorem]{Definition}

\theoremstyle{remark}
\newtheorem{remark}[theorem]{Remark}

\numberwithin{equation}{section}

\DeclareMathOperator*{\loc}{loc}

\DeclareMathOperator*{\Span}{span}
\DeclareMathOperator*{\supp}{supp}
\DeclareMathOperator*{\diag}{diag}
\DeclareMathOperator*{\esssup}{ess\,sup}

\DeclareMathOperator*{\Co}{Co}

\usepackage{xparse}

\newcommand{\WLwr}{\mathcal{W}(L^r_w)}

\newcommand{\Schwartz}{\mathcal{S}}
\newcommand{\SP}{\mathcal{S}' (\mathbb{R}^d) / \mathcal{P} (\mathbb{R}^d)}

\newcommand{\Indicator}{\mathds{1}}
\newcommand{\TL}{\dot{\mathbf{F}}^{\alpha}_{p,q}}
\newcommand{\TLi}{\dot{\mathbf{F}}^{\alpha}_{\infty,q}}
\newcommand{\TLii}{\dot{\mathbf{F}}^{\alpha}_{\infty,\infty}}
\newcommand{\Bii}{\dot{\mathbf{B}}^{\alpha}_{\infty,\infty}}
\NewDocumentCommand\TLseq{O{\alpha}}{\dot{\mathbf{f}}^{#1}_{p,q}}
\NewDocumentCommand\PT{O{\alpha}D<>{\beta}}{\dot{\mathbf{P}}^{#1, #2}_{p,q}}
\NewDocumentCommand\PTi{O{\alpha}D<>{\beta}}{\dot{\mathbf{P}}^{#1, #2}_{\infty,q}}
\NewDocumentCommand\PTii{O{\alpha}D<>{\beta}}{\dot{\mathbf{P}}^{#1, #2}_{\infty,\infty}}

\NewDocumentCommand\PTseq{O{\alpha}D<>{\beta}}{\dot{\mathbf{p}}^{#1, #2}_{\infty,q}}

\newcommand{\vertiii}[1]{{\left\vert \kern-0.25ex
                            \left\vert \kern-0.25ex
                              \left\vert #1\right\vert\kern-0.25ex
                            \right\vert \kern-0.25ex
                          \right\vert}}

\NewDocumentCommand\DoubleStar{O{\varphi}m}{#1_{#2,\beta}^{\ast\ast}}

\renewcommand{\emptyset}{\varnothing}
\newcommand{\PosPart}[1]{#1^+}
\newcommand{\NegPart}[1]{#1^-}

\newcommand{\CalP}{\mathcal{P}}
\newcommand{\Measure}{\mathrm{m}}
\newcommand{\Lebesgue}[1]{\Measure(#1)}

\DeclareFontFamily{U}{mathx}{\hyphenchar\font45}
\DeclareFontShape{U}{mathx}{m}{n}{
      <5> <6> <7> <8> <9> <10>
      <10.95> <12> <14.4> <17.28> <20.74> <24.88>
      mathx10
      }{}
\DeclareSymbolFont{mathx}{U}{mathx}{m}{n}
\DeclareFontSubstitution{U}{mathx}{m}{n}
\DeclareMathAccent{\widecheck}{0}{mathx}{"71}
\DeclareMathAccent{\wideparen}{0}{mathx}{"75}

\newcommand{\R}{\mathbb{R}}
\newcommand{\SC}{\mathcal{S}}

\newcommand{\N}{\mathbb{N}}
\newcommand{\Z}{\mathbb{Z}}
\newcommand{\PTalt}[1]{\dot{\mathbf{P}}^{#1, \beta}_{p,q}}
\newcommand{\PTalti}[1]{\dot{\mathbf{P}}^{#1, \beta}_{\infty,q}}
\newcommand{\PTaltii}[1]{\dot{\mathbf{P}}^{#1, \beta}_{\infty,\infty}}

\title[Anisotropic Triebel-Lizorkin spaces and wavelet coefficient decay, II]
      {Anisotropic Triebel-Lizorkin spaces and  wavelet coefficient decay over  one-parameter dilation groups, II}

\author[S. Koppensteiner]{Sarah Koppensteiner}
\address{Faculty of Mathematics,
University of Vienna,
Oskar-Morgenstern-Platz 1,
A-1090 Vienna, Austria}
\email{sarah.koppensteiner@univie.ac.at}

\author[J.T. van Velthoven]{Jordy Timo van Velthoven}
\address{Delft University of Technology,
Mekelweg 4, Building 36,
2628 CD Delft, The Netherlands}
\email{j.t.vanvelthoven@tudelft.nl}

\author[F. Voigtlaender]{Felix Voigtlaender}

\address{
Katholische Universit\"at Eichst\"att-Ingolstadt,
Mathematical Institute for Machine Learning and Data Science (MIDS),
Research group \emph{Reliable Machine Learning},
Ostenstrasse 26,
85072 Eichst\"att,
Germany
}
\email{felix.voigtlaender@ku.de}


\begin{document}

\maketitle

\begin{abstract}
Continuing previous work, this paper provides maximal characterizations of anisotropic Triebel-Lizorkin spaces $\TL$ for the endpoint case of $p = \infty$ and the full scale of parameters $\alpha \in \mathbb{R}$ and $q \in (0,\infty]$. In particular, a Peetre-type characterization of the anisotropic Besov space $\dot{\mathbf{B}}^{\alpha}_{\infty,\infty} = \dot{\mathbf{F}}^{\alpha}_{\infty,\infty}$
is obtained. As a consequence, it is shown that there exist dual molecular frames and Riesz sequences in $\TLi$.
\end{abstract}

\section{Introduction}
In a previous paper \cite{KvVV2021anisotropic}, we obtained characterizations of anisotropic Triebel-Lizorkin spaces $\TL$, with $\alpha \in \mathbb{R}$, $p \in (0,\infty)$ and $q \in (0,\infty]$, in terms of Peetre-type maximal functions and continuous wavelet transforms. In addition, as an application of these characterizations, it was shown that these spaces admit molecular dual frames and Riesz sequences.
The purpose of the present paper is to provide analogous results for the endpoint case of $p = \infty$.

For defining the anisotropic Triebel-Lizorkin spaces, let $A \in \mathrm{GL}(d, \mathbb{R})$ be an expansive matrix, i.e., $|\lambda|>1$ for all $\lambda \in \sigma(A)$, and let $\varphi \in \mathcal{S} (\mathbb{R}^d)$ be such that it has compact Fourier support
\begin{align}\label{eq:support1}
  \supp \widehat{\varphi}
  = \overline{
      \{
        \xi \in \mathbb{R}^d : \widehat{\varphi} (\xi) \neq 0
      \}
    }
  \subset \mathbb{R}^d \setminus \{ 0 \}
\end{align}
and satisfies
\begin{align}\label{eq:support2}
  \sup_{j \in \mathbb{Z}}
    | \widehat{\varphi} ((A^*)^j \xi) |
  > 0,
  \quad \xi \in \mathbb{R}^d \setminus \{0\},
\end{align}
with $A^*$ denoting the transpose of $A$.
For $j \in \mathbb{Z}$, denote its dilation by $\varphi_j = |\det A|^j \varphi (A^j \cdot)$.
The associated (homogeneous) \emph{anisotropic Triebel-Lizorkin space}
$\TLi = \TLi(A,\varphi)$, with $\alpha \in \mathbb{R}$ and $q \in (0, \infty]$,
is defined as the collection of all tempered distributions $f \in \mathcal{S}'(\mathbb{R}^d)$ (modulo polynomials) that satisfy
\begin{align} \label{eq:qfinite_intro}
  \| f \|_{\TLi}
  := \sup_{\ell \in \Z, k \in \Z^d} \bigg(
  \frac{1}{|\det A|^{\ell}} \int_{A^\ell ([0,1]^d + k)} \sum_{j=-\ell}^\infty \big( |\det A|^{\alpha j}  |(f \ast \varphi_j)(x)| \big)^q
   \, dx
     \bigg)^{1/q}
  < \infty
 \end{align}
 if $q < \infty$, and
\begin{align} \label{eq:qinfinite_intro}
  \| f \|_{\TLii} := \sup_{\ell \in \Z, k \in \Z^d} \sup_{j \in \Z, j
    \geq -\ell} \bigg( \frac{1}{|\det A|^{\ell}}
  \int_{A^\ell ([0,1]^d + k)} |\det A|^{\alpha j}  |(f \ast \varphi_j)(x)|
  \, dx \bigg ) < \infty.
\end{align}
Note that the $\ell^{\infty} (\mathbb{Z}_{\geq -\ell})$-norms in \Cref{eq:qinfinite_intro} are positioned outside of the integral,
whereas the $\ell^q(\mathbb{Z}_{\geq -\ell})$-norms in \Cref{eq:qfinite_intro} are part of the integrand.

In contrast to the usual quasi-norms defining (anisotropic) Triebel-Lizorkin spaces $\TL$ for $p < \infty$ (see, e.g., \cite{bownik2007anisotropic, bownik2008duality, bownik2006atomic, KvVV2021anisotropic}), the quantities \eqref{eq:qfinite_intro} and \eqref{eq:qinfinite_intro} consider only averages over small scales. The quasi-norms \eqref{eq:qfinite_intro} and \eqref{eq:qinfinite_intro} can therefore be considered as  \textquotedblleft localized versions" of the immediate analogue of the quasi-norms defining $\TL$ for $p< \infty$, which would lead to an unsatisfactory definition of $\TLi$, see \cite[Section 5]{frazier1990discrete} and the references therein.

The above definition of $\TLi$ follows Bownik \cite[Section 3]{bownik2007anisotropic} (see also \cite[Section 5]{frazier1990discrete}) for $q < \infty$,
but differs from \cite{bownik2007anisotropic, frazier1990discrete} for $q=\infty$, where $\TLii$ is instead defined via the quasi-norm
\begin{align} \label{eq:FB}
\| f \|_{\TLii} = \| f \|_{\Bii} :=  \sup_{j \in \mathbb{Z}} |\det A|^{\alpha j} \| f \ast \varphi_j \|_{L^{\infty}} < \infty,
\end{align}
with $\Bii$ denoting the (anisotropic) Besov space \cite{bownik2005atomic}.
The quasi-norm \eqref{eq:qinfinite_intro} is an anisotropic version of the definition given by Bui and Taibleson \cite{bui2000characterization},
which (as a consequence of our main results) will be shown to be equivalent to \eqref{eq:FB}, like for isotropic dilations \cite[Theorem 3]{bui2000characterization}.

\subsection{Maximal characterizations}
As in \cite{KvVV2021anisotropic}, we assume additionally that the expansive matrix $A \in \mathrm{GL}(d, \mathbb{R})$ is \emph{exponential}, in the sense that $A = \exp(B)$ for some matrix $B \in \mathbb{R}^{d \times d}$. The power of $A$ is then defined as $A^s := \exp(sB)$ for $s \in \mathbb{R}$.

For $\varphi \in \mathcal{S}(\mathbb{R}^d)$, $s \in \mathbb{R}$ and $\beta > 0$, the associated Peetre-type maximal function of $f \in \mathcal{S}' (\mathbb{R}^d)$ is defined by
\[
  \DoubleStar{s} f : \mathbb{R}^d \to [0, \infty], \quad x \mapsto
   \sup_{z \in \mathbb{R}^d}
       \frac{|f \ast \varphi_s (x + z)|}
            {(1 + \rho_A (A^{s} z))^{\beta}},
\]
where $\varphi_s := |\det A|^s \varphi (A^s \cdot)$ and where $\rho_A : \mathbb{R}^d \to [0,\infty)$
denotes the step homogeneous quasi-norm associated with $A$ (cf. \Cref{sub:expansive}).

The following Peetre-type maximal characterizations of $\TLi$ will be proven in Section \ref{sec:maximal}. Here, the notation $\dashint_Q$ means the average integral over a measurable set $Q \subseteq \mathbb{R}^d$ of positive measure.

\begin{theorem}\label{thm:main1_intro}
  Let $A \in \mathrm{GL}(d, \mathbb{R})$ be expansive and exponential and let $\alpha \in \mathbb{R}$.
  Assume that $\varphi \in \Schwartz(\R^d)$ has compact Fourier support
  satisfying the support conditions \eqref{eq:support1} and \eqref{eq:support2}.

  For  $q \in (0, \infty)$
  and $\beta >  1/q$, the norm equivalences
  \begin{equation}
    \begin{aligned}\label{eq:norm_equiv}
      \| f \|_{\TLi} & \asymp \sup_{\ell \in \Z, k \in \Z^d}
      \bigg(\; \dashint_{A^\ell
        ([0,1]^d + k)} \int_{-\ell}^\infty \big( |\det A|^{\alpha s}
       \DoubleStar{s} f(x) \big)^q \, ds \, dx
      \bigg)^{1/q} \\
      & \asymp \sup_{\ell \in \Z, k \in \Z^d}
      \bigg(\; \dashint_{A^\ell (
        [0,1]^d + k)} \sum_{j=-\ell}^\infty \big( |\det A|^{\alpha j}
       \DoubleStar{j} f(x) \big)^q \, dx \bigg)^{1/q}
    \end{aligned}
  \end{equation}
  hold for all $f \in \SP$. For $q = \infty$ and
  $\beta >1$, the following equivalences hold
  \begin{equation}
    \begin{aligned}\label{eq:norm_equiv_infty}
      \| f \|_{\TLii} & \asymp \sup_{\ell \in \Z, k \in \Z^d} \sup_{s
        \in \R, s \geq - \ell} \bigg( \dashint_{A^\ell ([0,1]^d + k)} |\det A|^{\alpha s}
      \DoubleStar{s} f(x) \, dx \bigg)
      \\
      & \asymp \sup_{\ell \in \Z, k \in \Z^d} \sup_{j \in \Z, j \geq -
        \ell} \bigg( \dashint_{A^\ell ([0,1]^d + k)} |\det A|^{\alpha j}
      \DoubleStar{j} f(x) \, dx  \bigg)
    \end{aligned}
  \end{equation}
  for all $f \in \SP$.
\end{theorem}

Theorem \ref{thm:main1_intro} provides an extension of \cite[Theorem 1]{bui2000characterization} to arbitrary expansive dilations,
and appears to be new even for the commonly studied setting of diagonal dilation matrices $A = \diag( \alpha_1, ..., \alpha_d)$ with given anisotropy $ (\alpha_1, ..., \alpha_d) \in (1,\infty)^d$.

The proof method of Theorem \ref{thm:main1_intro} is modeled on the proof of the maximal characterizations of Triebel-Lizorkin spaces $\TL$ with $p < \infty$ given in \cite{KvVV2021anisotropic}. In particular, it combines a sub-mean-value property of the Peetre-type maximal function (\Cref{prop:peetre-estimate}) with maximal inequalities. See also \cite{liu2019littlewood1, ullrich2012continuous, liang2012new, rychkov1991on} for similar approaches.
Besides the similarities in the approach, the calculations in the proof of Theorem \ref{thm:main1_intro} differ non-trivially from these in \cite[Theorem 3.5]{KvVV2021anisotropic} as only averages over small scales appear in the definition of $\TLi$.

As a consequence of \Cref{thm:main1_intro}, we show in  \Cref{sec:pqinfty} the coincidence $\TLii = \Bii$ mentioned above.

\subsection{Molecular decompositions}
For an expansive and exponential matrix $A \in \mathrm{GL}(d, \mathbb{R})$, denote by $G_A = \mathbb{R}^d \rtimes_A \mathbb{R}$ the associated semi-direct product group. Then $G_A$ acts unitarily on $L^2 (\mathbb{R}^d)$ via the quasi-regular representation $\pi$, defined by
\begin{align} \label{eq:dilate_tranlate}
  \pi(x,s) f = |\det A|^{-s/2} f(A^{-s} (\cdot - x)),
  \quad (x,s) \in \mathbb{R}^d \times \mathbb{R}, \; f \in L^2 (\mathbb{R}^d).
\end{align}
A vector $\psi \in L^2 (\mathbb{R}^d)$ is called \emph{admissible} if the associated wavelet transform
\[
W_{\psi} : L^2(\mathbb{R}^d) \to L^{\infty} (G_A), \quad W_{\psi} f = \langle f, \pi(\cdot) \psi \rangle,
\]
defines an isometry into $L^2 (G_A)$. The existence of admissible vectors and associated Calder\'on-type reproducing formulae for this representation have been studied, among others, in
\cite{laugesen2002characterization, currey2016integrable, groechenig1992compact, fuehr2002continuous}. The assumption that $A$ is expansive is essential for the existence of admissible vectors $\psi \in \mathcal{S} (\mathbb{R}^d)$ satisfying $\widehat{\psi} \in C^{\infty}_c (\mathbb{R}^d \setminus \{0\})$ (cf. \Cref{lem:admissible_expansive}), which play an important role in this paper.

A countable family $(\phi_{\gamma} )_{\gamma \in \Gamma}$ of functions $\phi_{\gamma} \in L^2 (\mathbb{R}^d)$ parametrized by a discrete set $\Gamma \subset G_A$ is called a \emph{molecular system} if there exists a function $\Phi \in \WLwr \subset L^1 (G_A)$ such that
\begin{align} \label{eq:molecule_def_intro}
|W_{\psi} \phi_{\gamma} (g) | = |\langle \phi_{\gamma} , \pi(g) \psi \rangle | \leq \Phi (\gamma^{-1} g), \quad \gamma \in \Gamma, \; g \in G.
\end{align}
The space $\WLwr$ denotes a weighted Wiener amalgam space for $r = \min\{1, q\}$ and the standard control weight $w = w_{\infty, q}^{\alpha', \beta} : G_A \to [1,\infty)$; see \Cref{sec:molecular} for further details.

It should be mentioned that any family $(\pi(\gamma) \phi)_{\gamma \in \Gamma}$ for suitable $\phi \in L^2 (\mathbb{R}^d)$ defines a molecular system in the sense of \eqref{eq:molecule_def} with $\Phi = |W_{\psi} \phi|$, but that generally a molecular system $(\phi_{\gamma} )_{\gamma \in \Gamma}$ does not need to consist of translates and dilates of a fixed function $\phi$. Nevertheless,
general molecules $(\phi_{\gamma} )_{\gamma \in \Gamma}$ share many properties with atoms $ (\pi(\gamma) \phi)_{\gamma \in \Gamma}$, see, e.g., \cite{grochenig2009molecules,velthoven2022quasi}.

The following theorem provides decomposition theorems of $\TLi$ in terms of molecules.

\begin{theorem} \label{thm:main2_intro}
Let $A \in \mathrm{GL}(d, \mathbb{R})$ be expansive and exponential. For $\alpha \in \mathbb{R}$, $q \in (0, \infty]$,  let $r := \min\{1, q\}$ and let $\alpha' = \alpha + 1/2 - 1/q$ if $q < \infty$ and $\alpha' = \alpha + 1/2$, otherwise. Let $\beta > 1/q$ if $q < \infty$ and $\beta > 1$ otherwise.

Suppose $\psi \in L^2 (\mathbb{R}^d)$ is admissible satisfying $W_{\psi} \psi \in \WLwr$ for the standard control weight $w = w_{\infty, q}^{- \alpha', \beta} : G_A \to [1,\infty)$ defined in \Cref{lem:ControlWeights}. Additionally, suppose that $W_{\varphi} \psi \in \WLwr$ for some (equivalently, all) admissible $\varphi \in \mathcal{S}_0 (\mathbb{R}^d)$.
Then there exists a compact unit neighborhood $U \subset G_A$ with the following property: For any discrete set $\Gamma \subset G_A$ satisfying
\begin{align} \label{eq:well-spread}
G_A = \bigcup_{\gamma \in \Gamma} \gamma U \quad \text{and} \quad  \sup_{g \in G_A} \# (\Gamma \cap g U) < \infty,
\end{align}
there exists a molecular system $(\phi_{\gamma})_{\gamma \in \Gamma}$
such that any $f \in \TLi$ admits the expansion
\[
f = \sum_{\gamma \in \Gamma} \langle f, \pi(\gamma) \psi \rangle \phi_{\gamma} = \sum_{\gamma \in \Gamma} \langle f,  \phi_{\gamma}  \rangle  \pi(\gamma) \psi,
\]
where the series converges unconditionally in the weak$^*$-topology of $\SP$.
\end{theorem}

Atomic decompositions of the anisotropic spaces $\TLi$ have been obtained earlier by Bownik \cite{bownik2007anisotropic}. However, Theorem \ref{thm:main2_intro} provides a frame decomposition of all elements $f \in \TLi$ in terms of the atoms $(\pi(\gamma) \psi)_{\gamma \in \Gamma}$ and molecules $(\phi_{\gamma})_{\gamma \in \Gamma}$, whereas the atoms in \cite[Theorem 5.7]{bownik2007anisotropic} depend on the element $f \in \TLi$ that is represented. For anisotropic Triebel-Lizorkin spaces $\TL$ with $p < \infty$, decompositions as in Theorem \ref{thm:main2_intro} were obtained in \cite{ho2003frames, KvVV2021anisotropic}, but they appear to be new for the case of $p = \infty$. In fact, Theorem \ref{thm:main2_intro} seems even valuable for merely isotropic dilations, where the state-of-the-art \cite{gilbert2002smooth} excludes the case $p = \infty$.

Theorem \ref{thm:main2_intro} will be obtained from the recent results on dual molecules \cite{velthoven2022quasi, romero2020dual} through the identification of $\TLi$ with a coorbit space; see \Cref{prop:TL_coorbit}. This identification appears to be new for the full scale of $\TLi$ with $\alpha \in \mathbb{R}$ and $q \in (0, \infty]$, even for isotropic dilations.

In addition to the existence of dual molecular frames, we also obtain a corresponding result for Riesz sequences. Here, the space $\dot{\mathbf{p}}^{- \alpha', \beta}_{\infty, q}$ denotes a sequence space associated to $\TLi$; see  \Cref{def:peetre_seq} for its precise definition.

\begin{theorem} \label{thm:main3_intro}
With assumptions and notations as in \Cref{thm:main2_intro}, the following holds:

There exists a compact unit neighborhood $U \subset G_A$ with the following property: For any discrete set $\Gamma \subset G_A$ satisfying
\begin{align} \label{eq:separated}
\gamma U \cap \gamma' U = \emptyset \quad \text{for all} \quad \gamma, \gamma' \in \Gamma \quad \text{with} \quad \gamma \neq \gamma',
\end{align}
there exists a molecular system $(\phi_{\gamma})_{\gamma \in \Gamma}$ in $\overline{\Span} \{ \pi(\gamma) \psi : \gamma \in \Gamma \}$
such that the distribution  $f := \sum_{\gamma \in \Gamma} c_{\gamma} \phi_{\gamma} \in \TLi$ forms a solution to the moment problem
\[
\langle f, \pi(\gamma) \psi \rangle = c_{\gamma}, \quad \gamma \in \Gamma
\]
for any given $(c_{\gamma} )_{\gamma \in \Gamma} \in \dot{\mathbf{p}}^{- \alpha', \beta}_{\infty, q} \leq \mathbb{C}^{\Gamma}$.
\end{theorem}

Riesz sequences in $\TLi$ seem not to have appeared in the literature before, which makes Theorem \ref{thm:main3_intro} new even for isotropic dilations.
Similarly to Theorem \ref{thm:main2_intro}, we obtain \Cref{thm:main3_intro} by applying results of \cite{velthoven2022quasi, romero2020dual}
to the coorbit realization of the Triebel-Lizorkin spaces $\TLi$.

\subsection*{Notation}
We denote by $\PosPart{s} := \max\{ 0, s \}$ and
$\NegPart{s} := - \min \{ 0, s \}$ the positive and negative part of
$s \in \mathbb{R}$.  If $f_1, f_2$ are positive functions on a common
base set $X$, the notation $f_1 \lesssim f_2$ is used to denote the
existence of a constant $C > 0$ such that $f_1 (x) \leq C f_2 (x)$ for
all $x \in X$. The notation $f_1 \asymp f_2$ is used whenever both
$f_1 \lesssim f_2$ and $f_2 \lesssim f_1$.  We will sometimes use
$\lesssim_\alpha$ to indicate that the implicit constant depends on
a quantity $\alpha$.

For a function $f : \mathbb{R}^d \to \mathbb{C}$ and a matrix $A \in \mathbb{R}^{d \times d}$, the dilation of  $f : \mathbb{R}^d \to \mathbb{C}$ is denoted by
  $f_j := |\det A|^j \, f(A^j \cdot)$, where $j \in \mathbb{Z}$.
 Similarly, we write $f_s := |\det A|^s \, f(A^s \cdot)$ for $s \in \mathbb{R}$, provided that $A^s$ is well-defined.

 The class of Schwartz functions on $\mathbb{R}^d$ will be denoted by
$\mathcal{S} (\mathbb{R}^d)$. Its dual space is simply denoted by $\mathcal{S}' (\mathbb{R}^d)$. Moreover, the notation
$\mathcal{P}(\mathbb{R}^d)$ will be used for the collection of polynomials on $\mathbb{R}^d$, and we write $\SP$ for the quotient space of tempered distributions modulo polynomials.  The Fourier transform
$\mathcal{F} : \mathcal{S}(\mathbb{R}^d) \to
\mathcal{S}(\mathbb{R}^d)$ is defined as
$\widehat{f} (\xi) = \int_{\mathbb{R}^d} f(x) e^{-2\pi ix \cdot \xi} \;
dx$ with inverse $\widecheck{f} := \mathcal{F}^{-1} f := \widehat{f}(- \, \cdot \,)$.
Similar notations will be used for the extension of the Fourier transforms to $L^2 (\mathbb{R}^d)$.

 The Lebesgue measure on $\mathbb{R}^d$ is denoted by $\Measure$. For a measurable set $Q \subset \mathbb{R}^d$ of finite, positive measure, it will be written $\dashint_{Q} f(x) \; dx := \Lebesgue{Q}^{-1} \int_{Q} f(x) \; dx$ for $f : \mathbb{R}^d \to \mathbb{C}$. The closure of a set $Q \subset \mathbb{R}^d$ is denoted by $\overline{Q}$.

 If $G$ is a group, then the left and right translation of a function $F : G \to \mathbb{C}$ by $h \in G$
will be denoted by
$
  L_h F = F(h^{-1} \cdot) $
  and $
  R_h F = F (\cdot \, h),
$ respectively. In addition, we write $F^{\vee} (x) = F(x^{-1})$.

\section{Anisotropic Triebel-Lizorkin spaces with $p = \infty$}
This section provides preliminaries on expansive matrices and Triebel-Lizorkin spaces.

\subsection{Expansive matrices} \label{sub:expansive}
A matrix $A \in \R^{d \times d}$ is said to be \emph{expansive}
if $| \lambda| > 1$ for all eigenvalues $\lambda \in \sigma(A)$,
where $\sigma(A)$ denotes the spectrum of $A$. Equivalently, a
matrix $A \in \R^{d \times d}$ is expansive if $\|A^{-j} \| \to 0$ as $j \to \infty$.

The following lemma collects several basic properties of expansive
matrices that will be used in the sequel, see, e.g.,
\cite[Definitions~2.3 and 2.5]{bownik2003anisotropic} and
\cite[Lemma~2.2]{bownik2003anisotropic}.  For a more general
background on spaces of homogeneous type on $\mathbb{R}^d$, we refer
to \cite{CoifmanSpacesOfHomogeneousType,coifman1977extensions}.

\begin{lemma}\label{lem:QuasiNorm}
Let $A \in \mathrm{GL}(d, \mathbb{R})$ be expansive.
\begin{enumerate}[(i)]
  \item There exist an ellipsoid $\Omega=\Omega_A$ (that is, $\Omega = P (B_1 (0))$ for some $P \in \mathrm{GL}(d, \mathbb{R})$) and $r > 1$ such that
        \[
          \Omega \subset r\Omega \subset A\Omega
        \]
        and $\Lebesgue{\Omega} = 1$.
        The function $\rho_A : \mathbb{R}^d \to [0,\infty)$ defined  by
        \begin{align}\label{eq:step_norm}
        \rho_A (x)
        = \begin{cases}
            |\det A|^j, \quad & \text{if} \;\; x \in A^{j+1} \Omega \setminus A^j \Omega, \\
            0, \quad          & \text{if} \;\; x = 0,
          \end{cases}
        \end{align}
        is Borel measurable and forms a quasi-norm, i.e., there exists $C \geq 1$ such that
        \begin{equation}
          \begin{alignedat}{3}
            \rho_A (-x)  & = \rho_A(x),
            && \quad x \in \mathbb{R}^d ,\\
            \rho_A (x)   &> 0,
            && \quad x \in \mathbb{R}^d \setminus \{0\} , \\
            \rho_A (Ax)  &= |\det A|  \rho_A (x),
            && \quad x \in \mathbb{R}^d , \\
            \rho_A (x+y) &\leq C \big( \rho_A (x) + \rho_A (y) \big),
            && \quad x,y \in \mathbb{R}^d.
          \end{alignedat}
          \label{eq:QuasiNormProperties}
        \end{equation}

  \item The function  $d_A : \R^d \times \R^d \to [0,\infty), (x,y) \mapsto \rho_A(x-y)$
        forms a quasi-metric.
        The triple $(\mathbb{R}^d, d_A, \Measure)$
        forms a space of homogeneous type.
\end{enumerate}
\end{lemma}

For an expansive matrix $A \in \mathrm{GL}(d, \mathbb{R})$, a function $\rho_A : \mathbb{R}^d \to [0,\infty)$ defined by Equation \eqref{eq:step_norm} will be called a \emph{step homogeneous quasi-norm} associated to $A$.
Given $y \in \mathbb{R}^d$ and $r>0$, its associated metric ball will be denoted by
\[ B_{\rho_A}(y, r) := \{ x \in \R^d \colon \rho_A (x-y) < r \}. \]
It is readily verified that $ B_{\rho_A}(0, 1) = \Omega$.  Hence, its metric balls are
of the form
\[ B_{\rho_A}(y, r) = A^{\ell} \Omega +y,
\]
where $\ell \in \Z$
is such that $|\det A|^{\ell-1} < r \leq |\det A|^{\ell} $.

Throughout this paper, given an expansive $A \in \mathrm{GL}(d, \mathbb{R})$, we will fix an ellipsoid $\Omega = \Omega_A$ as appearing in \Cref{lem:QuasiNorm} (i). This choice is not unique. Any other choice of ellipsoid will yield an equivalent quasi-norm, see, e.g., \cite[Lemma 2.4]{bownik2003anisotropic}.

\subsection{Schwartz seminorms}
 Let $A \in \mathrm{GL}(d, \mathbb{R})$ be expansive with associated step homogeneous quasi-norm $\rho_A$.
 For $\beta > 0$, the quasi-norm properties  \eqref{eq:QuasiNormProperties}
 imply that
  \begin{align} \label{eq:nusubmultiplicative}
  (1 + \rho_A (x + y))^{\beta} \lesssim_{A, \beta} (1 + \rho_A (x))^{\beta} (1 + \rho_A (y))^{\beta}
  \end{align}
  for $x,y \in \R^d$. For ease of notation, we will often write $\nu_{\beta} (x) = (1+\rho_A(x))^{\beta}$ for $x \in \R^d$.

 Let $\lambda_{-}$ and $\lambda_{+}$ be such that $1 < \lambda_{-} < \min_{\lambda \in \sigma(A)} |\lambda| \leq \max_{\lambda \in \sigma(A)} |\lambda| < \lambda_{+}$. Put
 \begin{align*} 
 \zeta_- := \frac{\ln \lambda_-}{\ln |\det A|} \in (0, d^{-1}) \quad \text{and} \quad \zeta_+ := \frac{\ln \lambda_+}{\ln |\det A|} \in (d^{-1}, \infty).
 \end{align*}
 Then \cite[Lemma 3.2]{bownik2003anisotropic} (see also \cite[Lemma~2.2]{KvVV2021anisotropic}) implies that
  \begin{equation*}
    \begin{alignedat}{3}
      \rho_A (x) &\lesssim \| x \|^{1/\zeta_-} + \| x \|^{1/\zeta_+},
      && \qquad x \in \R^d, \\
      \| x \| &\lesssim \rho_A(x)^{\zeta_-} + \rho_A(x)^{\zeta+},
      && \qquad x \in \R^d .
    \end{alignedat}
  \end{equation*}
  Therefore, the collection
  \begin{equation}
    \label{eq:semi-norms}
    p_{M,N}(\varphi):= \max_{ \substack{|\alpha| \leq M, \\ 0 < \beta \leq N}  } \sup_{y \in \R^d} \;
     (1 + \rho_A (y))^{\beta} \cdot |\partial^\alpha \varphi (y)|,
    \quad M,N \in \N,
  \end{equation}
  defines an equivalent family of seminorms for the Schwartz space $\mathcal{S} (\mathbb{R}^d)$.

\subsection{Analyzing vectors}
\label{sec:analyzing}
Let $A \in \mathrm{GL}(d, \mathbb{R})$ be expansive. Choose a function $\varphi \in \mathcal{S} (\mathbb{R}^d)$ with compact Fourier support
\begin{align}\label{eq:analyzing_support}
  \supp \widehat{\varphi}
  := \overline{\{ \xi \in \mathbb{R}^d : \widehat{\varphi}(\xi) \neq 0 \}}
  \subset \mathbb{R}^d \setminus \{0\}
\end{align}
satisfying, in addition,
\begin{align}\label{eq:analyzing_positive}
  \sup_{j \in \mathbb{Z}}
    \big| \widehat{\varphi} ((A^*)^j \xi) \big| > 0,
  \quad \xi \in \mathbb{R}^d \setminus \{0\}.
\end{align}
Then the function $\psi \in \SC(\R^d)$
defined by
\begin{equation*}
  \widehat\psi(\xi)
  = \begin{cases}
      \overline{\widehat\varphi(\xi)} / \sum_{k \in \Z}
                                          |\widehat{\varphi} ((A^*)^k \xi)|^2,
      \quad
      & \text{if} \;\; \xi \in \R^d \setminus \{0\}, \\
      0,
      \quad
      & \text{if} \;\; \xi = 0,
    \end{cases}
\end{equation*}
is well-defined and satisfies
\begin{align}\label{eq:analyzing_calderon}
  \sum_{j \in \mathbb{Z}}
    \widehat{\varphi} ((A^*)^j \xi) \, \widehat{\psi} ((A^*)^j \xi)
  = 1,
  \quad \xi \in \mathbb{R}^d \setminus \{0\}.
\end{align}
For more details and further properties, see, e.g., \cite[Lemma~3.6]{bownik2006atomic}

\subsection{Anisotropic Triebel-Lizorkin spaces}
\label{sub:TLSpaces}

Let $A \in \mathrm{GL}(d, \mathbb{R})$ be expansive and
fix an analyzing vector $\varphi \in \Schwartz(\R^d)$
with compact Fourier support satisfying
\eqref{eq:analyzing_support} and \eqref{eq:analyzing_positive}.

For $\alpha \in \mathbb{R}$ and $0 < q < \infty$,
the associated (homogeneous) \emph{anisotropic Triebel-Lizorkin space} $\TLi = \TLi(A,\varphi)$
is defined as the collection of all $f \in \SP$ for which
\[
  \| f \|_{\TLi}
  := \sup_{\ell \in \Z, k \in \Z^d} \bigg(
  \dashint_{A^\ell ([0,1]^d + k)} \sum_{j=-\ell}^\infty \big( |\det A|^{\alpha j}  |(f \ast \varphi_j)(x)| \big)^q
   \, dx
     \bigg)^{1/q}
  < \infty.
 \]
The definition of $\TLi = \TLi(A, \varphi)$ is independent of the choice of analyzing vector $\varphi$, with equivalent quasi-norms for different choices, cf. \cite[Corollary 3.13]{bownik2007anisotropic}. In addition,
the space $\TLi$ is continuously embedded into $\mathcal{S}' (\mathbb{R}^d) / \mathcal{P}(\mathbb{R}^d)$, and is complete with respect to $\| \cdot \|_{\TLi}$. See \cite[Corollary 3.14]{bownik2007anisotropic} for both claims.

For the case $q = \infty$, the space $\TLii = \TLii(A, \varphi)$ will be defined as the collection of all $f \in \mathcal{S}' (\mathbb{R}^d) / \mathcal{P} (\mathbb{R}^d)$ satisfying
\[
  \| f \|_{\TLii} := \sup_{\ell \in \Z, k \in \Z^d} \sup_{j \in \Z, j
    \geq -\ell} \bigg(
  \dashint_{A^\ell ([0,1]^d + k)} |\det A|^{\alpha j}  |(f \ast \varphi_j)(x)|
  \, dx \bigg ) < \infty.
\]

For our purposes, it will be convenient to use the metric ball $\Omega = B_{\rho_A} (0,1)$ instead of the cube $[0,1]^d$ in defining $\TLi$. The independence of this choice is guaranteed by the following lemma, whose simple proof follows from a standard covering argument and is hence omitted.

\begin{lemma} \label{lem:independence_balls}
Let $F : \mathbb{R}^d \to [0, \infty)$ and $F_j :  \mathbb{R}^d \to [0, \infty)$, $j \in \mathbb{Z}$, be measurable functions. Then
\[
\sup_{\ell \in \mathbb{Z}, k \in \Z^d} \dashint_{A^\ell ([0,1]^d + k)} F(x) \; dx \asymp \sup_{\ell \in \mathbb{Z}, w \in \R^d} \dashint_{A^\ell \Omega + w} F(x)  \; dx
\]
and
\[
\sup_{\ell \in \mathbb{Z}, k \in \Z^d} \sup_{j \in \Z, j \geq - \ell} \; \dashint_{A^\ell ([0,1]^d + k)} F_j(x) \; dx \asymp \sup_{\ell \in \mathbb{Z}, w \in \R^d} \sup_{j \in \Z, j \geq -\ell} \; \dashint_{A^\ell \Omega + w} F_j(x) \; dx
\]
with implicit constants only depending on $d, A$.
\end{lemma}

\section{Maximal function characterizations} \label{sec:maximal}
Throughout this section, $A \in \mathrm{GL}(d, \mathbb{R})$ will denote an expansive matrix.

\subsection{Peetre-type maximal function}
Let $\varphi \in \mathcal{S}(\R^d)$. For $j \in \mathbb{Z}$ and $\beta > 0$,
the associated \emph{Peetre-type maximal function} of $f \in \mathcal{S}' (\mathbb{R}^d)$ is defined as
\begin{align}\label{eq:peetre_maximal}
  \DoubleStar{j} f (x)
  = \sup_{z \in \mathbb{R}^d}
       \frac{|(f \ast \varphi_{j}) (x+z)|}
            {(1 + \rho_{A}(A^j z))^{\beta}},
  \quad x \in \mathbb{R}^d,
\end{align}
where $\rho_A : \R^d \to [0,\infty)$ denotes the step homogeneous quasi-norm.

The Peetre-type maximal function has the following basic properties.

\begin{lemma}
  \label{lem:peetre-max-finitness}
  Let $\varphi \in \Schwartz(\R^d)$. Let $j \in \Z$ and $\beta > 0$.  For any $f \in \Schwartz ' (\R^d)$, the following holds:
  \begin{enumerate}[(i)]
  \item If $\DoubleStar{j} f (x)$ is finite for some $x \in \R^d$,
    then it is finite for all $x \in \R^d$.
  \item There
    exists an $N = N (f , A) \in \N$ such that for all $\beta \geq N$
    \begin{equation*}
    \varphi_{j,\beta}^{\ast\ast}f(x) < \infty, \quad x \in \R^d .
  \end{equation*}
  \end{enumerate}

\end{lemma}

\begin{proof}
  (i) Let $x_0 \in \mathbb{R}^d$ be such that $\DoubleStar{j} f (x_0) < \infty$. The symmetry of $\rho_A$ and
  the inequality \eqref{eq:nusubmultiplicative} for $\nu_\beta (y) = (1 + \rho_{A}(y))^{\beta}$ yield that
  \begin{align*}
    \DoubleStar{j} f (x) &= \sup_{z \in \R^d} \frac{|(f \ast \varphi_j) (z)|}
    {\nu_\beta\big(A^j(x-z)\big)}
    \lesssim \nu_\beta\big(A^j(x-x_0)\big) \sup_{z \in \R^d} \frac{|(f \ast \varphi_j) (z)|}
    {\nu_\beta\big(A^j(x_0-z)\big)} \\
    &= \nu_\beta\big(A^j(x-x_0)\big)
    \, \DoubleStar{j} f (x_0) < \infty.
  \end{align*}
 for arbitrary $x \in \mathbb{R}^d$.
\\~\\ (ii) Fix $x, z \in \R^d$. Let $\langle \cdot, \cdot \rangle_{\Schwartz', \Schwartz}$ denote the bilinear dual pairing. Since $f \in \Schwartz ' (\R^d)$ and the seminorms \eqref{eq:semi-norms} are
  ordered, there exist $M=M(f), N=(f,A) \in \N$ and a constant $C=C(f) > 0$ such
  that
  \begin{align*}
    |(f \ast \varphi_j) (z)| & = | \langle f, T_z  \varphi_j ^{\vee} \rangle_{\Schwartz', \Schwartz}| \\
    &\leq C \max_{ \substack{|\alpha| \leq M, \\ 0 < \beta \leq N}  }  \sup_{y \in \R^d}
     \nu_\beta (y) \, |\partial^\alpha \big(|\varphi_j^{\vee}|\big)(y-z)| \\
   & \lesssim C \max_{ \substack{|\alpha| \leq M, \\ 0 < \beta \leq N}  }  \sup_{\tilde{y} \in \R^d}
      \nu_\beta (\tilde{y} + x)  \nu_\beta (z - x) \,
     |\partial^\alpha \big(|\varphi_j^{\vee}|\big)(\tilde{y})|  \\
   & \lesssim C \max_{ \substack{|\alpha| \leq M, \\ 0 < \beta \leq N}  }  \sup_{\tilde{y} \in \R^d}
      \nu_\beta (\tilde{y} + x)  \max\{1, |\det A|^{-j \beta}\} \,
     \nu_\beta\big( A^j (z - x) \big) \,
     |\partial^\alpha \big(|\varphi_j^{\vee}|\big)(\tilde{y})| .
  \end{align*}
 Consequently,
  \begin{align*}
    \frac{|(f \ast \varphi_j) (z)|}{\nu_{N} \big( A^j (z - x) \big)}
    &\lesssim  \max_{ \substack{|\alpha| \leq M, \\ 0 < \beta \leq N}  }  \sup_{\tilde{y} \in \R^d}
    \nu_\beta (\tilde{y} + x)  \,
    |\partial^\alpha \big(|\varphi_j^{\vee}|\big)(\tilde{y})|
    \lesssim p_{M,N} (\varphi_j^{\vee}) .
  \end{align*}
 Since the constants are
  independent of $z \in \R^d$, the claim follows easily.
\end{proof}

\subsection{Sub-mean-value property} The following type of result is often referred to as a \textquotedblleft sub-mean-value property" and will play an essential role in deriving the main results. It forms an anisotropic analogue of the isotropic result \cite[Theorem~5]{stroemberg1989weighted}.

\begin{proposition}
  \label{prop:peetre-estimate}
  Let $A \in \mathrm{GL}(d, \mathbb{R})$ be expansive and let
  $\varphi \in \Schwartz(\R^d)$ have compact Fourier support satisfying \eqref{eq:analyzing_support}
  and \eqref{eq:analyzing_positive}.  Then, for all
  $q \in (0,\infty)$ and $\beta > 0$, there exists a constant
  $C= C(A,\varphi, q, \beta) > 0$ such that
  \begin{equation}
    \label{eq:peetre-estimate}
    (\DoubleStar{j} f (x))^q
    \leq C |\det A|^ j \int_{\R^d}  \frac{|(f \ast \varphi_j)(y)|^q}{\big(1+\rho_A(A^j(x-y))\big)^{\beta q}}
     \, dy, \quad x \in \R^d,
  \end{equation}
  for all $f \in \Schwartz'(\R^d)$ and all $j \in \Z$.
\end{proposition}

\begin{proof}
The claim is trivial whenever $\DoubleStar{j} f(x) = 0$, so we assume throughout that $\DoubleStar{j} f(x) > 0$.
The proof will
  be split into two steps dealing with $q \in [1,\infty)$ and
  $q \in (0, 1)$ separately.
  \\~\\
  \textbf{Step 1.} \emph{(The case $q \in [1,\infty)$).}
  By the compact Fourier support condition
  \eqref{eq:analyzing_support} and the identity
  \eqref{eq:analyzing_calderon}, it follows that there exists
  $N \in \N$ (depending on $\varphi$ and $A$) such that the function
  $\Phi := \sum_{k = -N}^N \varphi_{k} \ast \psi_{k}$
  satisfies
  \begin{equation}
    \label{eq:conv-identity}
    \varphi_j \ast \Phi_j = \varphi_j , \quad j \in \Z,
  \end{equation}
  see, e.g., the proof of \cite[Theorem~3.5]{KvVV2021anisotropic} for a detailed verification.

  Using that $\Phi_j = \sum_{k = -N}^N \varphi_{j + k} \ast \psi_{j + k}$ for $j \in \mathbb{Z}$, the equality \eqref{eq:conv-identity} gives
  \begin{align*}
    \frac{|(f \ast \varphi_j)(x+z)|}{\nu_{\beta}(A^j z)}
    &\leq \sum_{k = -N}^N \frac{|(f \ast \varphi_j \ast \varphi_{j+k}
      \ast \psi_{j+k})(x+z)|}{\nu_{\beta}(A^j z)} \\
    &\leq \sum_{k = -N}^N \int_{\R^d} \frac{|(f \ast \varphi_j)(y)|}{\nu_{\beta}(A^j z)}
      |(\varphi_{j+k} \ast \psi_{j+k})(x+z-y)| \, dy \\
    &\lesssim \sum_{k = -N}^N \int_{\R^d} \frac{|(f \ast \varphi_j)(y)|}
      {\nu_{\beta}\big(A^j (x-y)\big)}
       \nu_{\beta}\big(A^j (x+z-y)\big) \,
      |(\varphi_{j+k} \ast \psi_{j+k})(x+z-y)| \, dy
  \end{align*}
   for all $x, z \in \R^d$, $j \in \Z$ and
  $\beta > 0$. Combining
  H\"older's inequality for $\tfrac{1}{q}+\tfrac{1}{q'} = 1$ and the
  translation invariance of $L^{q'}(\R^d)$ yields
  \begin{align*}
    \frac{|(f \ast \varphi_j)(x+z)|}{\nu_{\beta}(A^j z)}
    &\lesssim \sum_{k=-N}^N \bigg\| \frac{(f \ast \varphi_j)}
      {\nu_{\beta}\big(A^j (x- \mkern2mu \cdot \mkern2mu)\big)} \bigg\|_{L^q}
      \Big\| \nu_{\beta}\big(A^j (x+z- \mkern2mu \cdot \mkern2mu)\big) \,
      (\varphi_{j+k} \ast \psi_{j+k})(x+z- \mkern2mu \cdot \mkern2mu) \Big\|_{L^{q'}} \\
    &= \sum_{k=-N}^N \bigg\| \frac{(f \ast \varphi_j)}
      {\nu_{\beta}\big(A^j (x- \mkern2mu \cdot \mkern2mu)\big)} \bigg\|_{L^q}
      \Big\| \nu_{\beta}\big(A^j (\, \cdot \,)\big) \,
      (\varphi_{j+k} \ast \psi_{j+k})(\, \cdot \,) \Big\|_{L^{q'}}.
  \end{align*}
If $q \in (1,\infty)$,  applying
  the transformation $A^j y \mapsto \tilde{y}$ in the $L^{q'}$-norm above gives
  \begin{align*}
     \Big\| \nu_{\beta}\big(A^j (\, \cdot \,)\big) \,
    (\varphi_{j+k} \ast \psi_{j+k})(\, \cdot \,) \Big\|_{L^{q'}}^{q'}
    & = \int_{\R^d} \big(\nu_{\beta}(A^j y )\big)^{q'} \, |\det A|^{j q'} \,
      |(\varphi_{k} \ast \psi_{k})(A^j y)|^{q'} \, dy \\
     & =  |\det A|^{j q'-j}  \int_{\R^d} \big(\nu_{\beta}( \tilde{y} )\big)^{q'} \,
    |(\varphi_{k} \ast \psi_{k})( \tilde{y})|^{q'} \, d\tilde{y} .
  \end{align*}
 Using the seminorms defined in Equation \eqref{eq:semi-norms} and the fact that
  $\varphi, \psi \in \Schwartz(\R^d)$, the last integral can be estimated by
  a constant $C_k = C_k(A,\varphi, q, \beta) > 0$.  Combining the above gives
  \begin{align*}
    \frac{|(f \ast \varphi_j)(x+z)|}{\nu_{\beta}(A^j z)}
    & \lesssim  \bigg\| \frac{(f \ast \varphi_j)}
      {\nu_{\beta}\big(A^j (x- \,\cdot\,)\big)} \bigg\|_{L^q}
      \sum_{k=-N}^N |\det A|^{j(1-1/q')} \, C_k^{1/q'} \\
    & \lesssim |\det A|^{j/q} \bigg\| \frac{(f \ast \varphi_j)}
      {\nu_{\beta}\big(A^j (x- \,\cdot\,)\big)} \bigg\|_{L^q} ,
  \end{align*}
  with implicit constant depending on $A, \beta, q, \varphi$.  Taking the supremum over
  $z \in \R^d$ and the $q$-th power yields the claim for $q \in (1, \infty)$.
  The case  $q = 1$ follows by the same arguments
  with the usual modifications.
  \\~\\
  \textbf{Step 2.} \emph{(The case $q \in (0, 1)$).} For $f \in \Schwartz' (\R^d)$, the estimate obtained in Step 1 (for $q=1$) gives
  \begin{align*}
    \DoubleStar{j} f (x)
    & \lesssim |\det A|^ j \int_{\R^d}  \frac{|(f \ast \varphi_j)(y)|}{\nu_\beta\big(A^j(x-y)\big)}
      \, dy \\
    & = |\det A|^ j \int_{\R^d} \bigg( \frac{|(f \ast \varphi_j)(y)|}{\nu_\beta\big(A^j(x-y)\big)} \bigg)^{1-q}
       \bigg( \frac{|(f \ast \varphi_j)(y)|}{\nu_\beta\big(A^j(x-y)\big)} \bigg)^{q}
      \, dy \\
    & \leq (\DoubleStar{j} f (x))^{1-q} |\det A|^ j
      \int_{\R^d}\bigg( \frac{|(f \ast \varphi_j)(y)|}
      {\nu_\beta\big(A^j(x-y)\big)} \bigg)^{q}
      \, dy . \numberthis \label{eq:peetre-est-eq1}
  \end{align*}
  By Lemma~\ref{lem:peetre-max-finitness}, there exists
  $N' = N'(f) \in \N$ such that $\DoubleStar{j} f (x) < \infty$ for all
  $\beta \geq N'$.  Hence, if $\beta \geq N'$, the claim follows
  immediately from the inequality \eqref{eq:peetre-est-eq1}.

  For the case
  $\beta < N'$, we use the already proven result for $N'$ to obtain
  \begin{align*}
    |(f \ast \varphi_j) (z)|^q
    & \lesssim |\det A|^ j \int_{\R^d}  \bigg( \frac{|(f \ast \varphi_j)(y)|}
      {\nu_{N'}\big(A^j(z-y)\big)} \bigg)^q
       \, dy \\
    & \leq |\det A|^ j \int_{\R^d}  \bigg(\frac{|(f \ast \varphi_j)(y)|}
      {\nu_\beta\big(A^j(z-y)\big)} \bigg)^q
      \, dy \\
    & \lesssim |\det A|^ j \int_{\R^d}  \bigg(\frac{|(f \ast \varphi_j)(y)|}
      {\nu_\beta\big(A^j(x-y)\big)} \bigg)^q \, \nu_\beta\big(A^j(x-z)\big)^q
       \, dy,
  \end{align*}
  where the second inequality used that $\nu_{N'}(x)^{-q} \leq \nu_\beta(x)^{-q}$ for
  $0 < \beta < N'$.  Consequently,
  \begin{equation*}
     \bigg( \frac{|(f \ast \varphi_j)(z)|}
     {\nu_\beta\big(A^j(x-z)\big)} \bigg)^q
     \lesssim |\det A|^ j \int_{\R^d}  \bigg(\frac{|(f \ast \varphi_j)(y)|}
      {\nu_\beta\big(A^j(x-y)\big)} \bigg)^q
       \, dy .
  \end{equation*}
  The right-hand side being independent of $z \in \R^d$, taking
  the supremum yields the claim for $\beta < N'$. Overall, this completes the proof.
\end{proof}

\subsection{Maximal function characterizations}
In this section, the matrix $A \in \mathrm{GL}(d, \mathbb{R})$ is often
additionally assumed to be \emph{exponential}, i.e., it is assumed
that $A$ admits the form $A = \exp(B)$ for some
$B \in \mathbb{R}^{d\times d}$. Then the power $A^s = \exp(s B)$ is
well-defined for $s \in \mathbb{R}$.

For $\varphi \in \Schwartz (\mathbb{R}^d)$ and
$s, \beta \in \mathbb{R}$ with $\beta > 0$, the associated Peetre-type
maximal function of $f \in \Schwartz' (\mathbb{R}^d)$ is defined as in
\eqref{eq:peetre_maximal} by
\[
  \DoubleStar{s} f(x) = \sup_{z \in \mathbb{R}^d} \frac{|(f \ast
    \varphi_s)(x+z)|}{(1+\rho_A(A^s z))^{\beta}}, \quad x \in
  \mathbb{R}^d
\]
whenever $A$ is exponential.

The following theorem forms a main result of this paper. It
characterizes the anisotropic Triebel-Lizorkin spaces $\TLi$, with
$0 < q \leq \infty$, in terms of Peetre-type maximal functions. The
result forms an extension of \cite[Theorem 1]{bui2000characterization}
to possibly anisotropic dilations.

\begin{theorem}\label{thm:norm_equiv}
  Suppose $A \in \mathrm{GL}(d, \mathbb{R})$ is expansive and exponential.
  Assume that $\varphi \in \Schwartz(\R^d)$ has compact Fourier support
  and satisfies \eqref{eq:analyzing_support} and \eqref{eq:analyzing_positive}.

  Then, for all $q \in (0, \infty)$, $\alpha \in \mathbb{R}$
  and $\beta >  1/q$, the norm equivalences
  \begin{equation}
    \begin{aligned}\label{eq:norm_equiv}
      \| f \|_{\TLi} & \asymp \sup_{\ell \in \Z, w \in \R^d}
      \bigg( \dashint_{A^\ell
        \Omega + w} \int_{-\ell}^\infty \big( |\det A|^{\alpha s}
       \DoubleStar{s} f(x) \big)^q \, ds \, dx
      \bigg)^{1/q} \\
      & \asymp \sup_{\ell \in \Z, w \in \R^d}
      \bigg( \dashint_{A^\ell
        \Omega + w} \sum_{j=-\ell}^\infty \big( |\det A|^{\alpha j}
       \DoubleStar{j} f(x) \big)^q \, dx \bigg)^{1/q}
    \end{aligned}
  \end{equation}
  hold for all $f \in \SP$. For $q = \infty$, $\alpha \in \R$ and
  $\beta >1$, the following equivalences hold
  \begin{equation}
    \begin{aligned}\label{eq:norm_equiv_infty}
      \| f \|_{\TLii} & \asymp \sup_{\ell \in \Z, w \in \R^d} \sup_{s
        \in \R, s \geq - \ell} \; \dashint_{A^\ell \Omega + w} |\det A|^{\alpha s}
      \DoubleStar{s} f(x) \, dx
       \\
      & \asymp \sup_{\ell \in \Z, w \in \R^d} \sup_{j \in \Z, j \geq -
        \ell} \; \dashint_{A^\ell \Omega + w} |\det A|^{\alpha j}
      \DoubleStar{j} f(x) \, dx  .
    \end{aligned}
  \end{equation}
  for all $f \in \SP$.

  (The function $\DoubleStar{s}f : \mathbb{R}^d \to [0,\infty]$  is well-defined
  for $f \in \SP$, since $\varphi$ has infinitely many vanishing moments and hence
  $P \ast \varphi_s = 0$ for every $P \in \CalP(\R^d)$.)
\end{theorem}

\begin{remark} \label{rem:general_expansive}
The proof of \Cref{thm:norm_equiv} shows that the discrete characterizations
    \begin{align*}
      \| f \|_{\TLi} \asymp \sup_{\ell \in \Z, w \in \R^d}
      \bigg( \dashint_{A^\ell
        \Omega + w} \sum_{j=-\ell}^\infty \big( |\det A|^{\alpha j}
       \DoubleStar{j} f(x) \big)^q \, dx \bigg)^{1/q}, \quad f \in \SP
    \end{align*}
and
    \begin{align*}
      \| f \|_{\TLii}  \asymp \sup_{\ell \in \Z, w \in \R^d} \sup_{j \in \Z, j \geq -
        \ell} \; \dashint_{A^\ell \Omega + w} |\det A|^{\alpha j}
      \DoubleStar{j} f(x) \, dx, \quad f \in \SP,
    \end{align*}
also hold without the assumption that $A \in \mathrm{GL}(d, \mathbb{R})$ is exponential.
\end{remark}

\begin{proof}[Proof of \Cref{thm:norm_equiv}]
  Only the cases $q \in (0,\infty)$ will be treated; the case
  $q = \infty$ follows by the arguments for $q = 1$, with the usual
  modification to accommodate the supremum.  The proof is split into
  three steps and for some parts we refer to calculations from the proof of
  \cite[Theorem~3.5]{KvVV2021anisotropic}.

  Throughout the proof, we will make use of the equivalent norms
  \[
  \| f \|_{\TLi} \asymp
   \sup_{\ell \in \Z, w \in \R^d} \bigg(
  \dashint_{A^\ell \Omega + w} \sum_{j=-\ell}^\infty \big( |\det A|^{\alpha j}  |(f \ast \varphi_j)(x)| \big)^q
   \, dx
     \bigg)^{1/q}
  \]
  and
   \[
  \| f \|_{\TLii} \asymp
   \sup_{\ell \in \Z, w \in \R^d} \sup_{j \in \mathbb{Z}, j \geq - \ell}
  \dashint_{A^\ell \Omega + w}  |\det A|^{\alpha j}  |(f \ast \varphi_j)(x)|
   \, dx
  \]
  provided by \Cref{lem:independence_balls}.
  \\~\\
  \textbf{Step 1.} In this step it will be shown that $\| f \|_{\TLi}$
  can be bounded by the middle term of \eqref{eq:norm_equiv}.  This
  step is modeled on Step 1 of the proof of \cite[Theorem
  3.5]{KvVV2021anisotropic}.  By the calculations constituting
  \cite[Equations (3.7)--(3.11)]{KvVV2021anisotropic}, it follows that,
  for $t \in [0,1]$, there exists $N = N(A, \varphi) \in \mathbb{N}$
  such that
  \begin{align*}
    \Big\| \!
    \Big(
    |\det A|^{\alpha j}  |(f \ast \varphi_j)(x)|
    \Big)_{j = - \ell}^\infty
    \Big\|_{\ell^q} \!
    & \lesssim \sum_{k =-N}^N
      \bigg\|
      \bigg(
      |\det A|^{\alpha(j+k+t)} \,
       \DoubleStar{j+k+t} f (x)
      \bigg)_{j = - \ell}^\infty
      \bigg\|_{\ell^q} \\
    & = \sum_{k =-N}^N
      \bigg\|
      \bigg(
      |\det A|^{\alpha(j+t)} \,
       \DoubleStar{j+t} f (x)
      \bigg)_{j = k - \ell}^\infty
      \bigg\|_{\ell^q} \\
    &\lesssim \bigg\|
      \bigg(
      |\det A|^{\alpha(j+t)}
       \DoubleStar{j+t} f (x)
      \bigg)_{j = - \ell-N}^\infty
      \bigg\|_{\ell^q}.
      \numberthis \label{eq:LHSindependt}
 \end{align*}
  If $q < \infty$, then raising \eqref{eq:LHSindependt} to the $q$-th
 power and integrating over $t \in [0,1]$ gives
 \begin{align*}
   \sum_{j = - \ell}^\infty \Big( |\det A|^{\alpha j}  |(f \ast \varphi_j)(x)| \Big)^q
   &\lesssim \int_0^1 \sum_{j = - (\ell+N)}^\infty \Big(|\det A|^{\alpha(j+t)}
      \DoubleStar{j+t} f (x) \Big)^q \, dt \\
   &= \int_{- (\ell + N)}^\infty \Big(|\det A|^{\alpha s}
      \DoubleStar{s} f (x) \Big)^q \, ds.
 \end{align*}
Let $w \in \R^d$ be arbitrary and set $Q_\ell = A^\ell \Omega + w$.
 By Lemma~\ref{lem:QuasiNorm}, it follows that $ Q_\ell \subset Q_{\ell +N}$. Therefore,
 averaging over $Q_{\ell}$ gives
 \begin{align*}
    \dashint_{Q_\ell}
     \sum_{j = - \ell}^\infty \Big( |\det A|^{\alpha j}
     |(f \ast \varphi_j)(x)| \Big)^q \, dx
   & \lesssim \dashint_{Q_\ell}
     \int_{- (\ell + N)}^\infty \!\Big(|\det A|^{\alpha s}
      \DoubleStar{s} f (x) \Big)^q \, ds \, dx \\
   & \leq \frac{\Lebesgue{ Q_{\ell+N}}}{\Lebesgue{Q_\ell}}
    \dashint_{Q_{\ell + N}}
     \int_{- (\ell + N)}^\infty \!\Big(|\det A|^{\alpha s}
      \DoubleStar{s} f (x) \Big)^q \, ds \, dx ,
 \end{align*}
 where
 $\frac{\Lebesgue{ Q_{\ell+N}}}{\Lebesgue{Q_\ell}} = |\det A|^N \lesssim 1$.  Consequently,
 taking  the $q$-th root and the supremum over
 $\ell \in \Z$ and $w \in \R^d$ yields the desired estimate.
 \\~\\
 \textbf{Step 2.} In this step we estimate the middle term by the
 right-most term of \eqref{eq:norm_equiv}.  This requires discretizing
 the inner-most integral, which works analogously to Step 2 in the proof of
 \cite[Theorem~3.5]{KvVV2021anisotropic}.  By
 \cite[Equation~(3.15)]{KvVV2021anisotropic}, for
 $t \in [0,1]$, there exists $N = N(A, \varphi) \in \mathbb{N}$ such that
  \begin{align*}
    \Big(
      |\det A|^{\alpha (j+t)}  \DoubleStar{j+t} f (x)
    \Big)^q
    &\lesssim  \sum_{k = -N}^N
               \Big(
                 |\det A|^{\alpha (j + k)}
                  \DoubleStar{j+k} f (x)
               \Big)^q.
    \numberthis \label{eq:det_Peetre}
  \end{align*}
  Starting with the inner-most integral of the middle term in
  \eqref{eq:norm_equiv}, we use a simple periodization argument and
  \eqref{eq:det_Peetre} to obtain
  \begin{align*}
    \int_{-\ell}^\infty \Big(
    |\det A|^{\alpha s}  \DoubleStar{s} f (x)
    \Big)^q \, ds
    & = \sum_{j= - \ell}^\infty \int_{0}^1 \Big(
      |\det A|^{\alpha (j+t)}  \DoubleStar{j+t} f (x)
      \Big)^q \, dt \\
    & \lesssim \sum_{j= - \ell}^\infty \sum_{k= -N}^N \Big(
      |\det A|^{\alpha (j+k)}  \DoubleStar{j+k} f (x)
      \Big)^q \\
    & \lesssim \sum_{j= - (\ell+N)}^\infty  \Big(
      |\det A|^{\alpha j}  \DoubleStar{j} f (x)
      \Big)^q  .
  \end{align*}
  Taking the averaged integral over $Q_\ell = A^\ell \Omega + w$ yields
  \begin{align*}
     \dashint_{Q_\ell} \int_{-\ell}^\infty \Big(
    |\det A|^{\alpha s}  \DoubleStar{s} f (x)
    \Big)^q   ds \, dx
    & \lesssim \dashint_{Q_\ell}
      \sum_{j= - (\ell+N)}^\infty  \Big(
      |\det A|^{\alpha j}  \DoubleStar{j} f (x)
      \Big)^q  \, dx \\
    & \lesssim \dashint_{Q_{\ell+N}}
      \sum_{j= - (\ell+N)}^\infty  \Big(
      |\det A|^{\alpha j}  \DoubleStar{j} f (x)
      \Big)^q  \, dx,
  \end{align*}
  where we used $Q_\ell \subset Q_{\ell+N}$ and
  $\frac{1}{\Lebesgue{Q_\ell}} = |\det A|^N
  \frac{1}{\Lebesgue{Q_{\ell+N}}}$ in the last step.  Taking the
  supremum over all $\ell \in \Z$ and $w \in \R^d$ and the $q$-th root
  yields the claim for $q \in (0,\infty)$.
 \\~\\
  \textbf{Step 3.} Lastly, it will be shown that the right-most term of
  \eqref{eq:norm_equiv} can be bounded by $\|f\|_{\TLi{}}$.
  We start with using Proposition~\ref{prop:peetre-estimate} for the
  exponent $0< q/r<\infty$, where $r:=\sqrt{\beta q} >1$ by
  assumption.  This gives for all $x \in \R^d$ and $\ell \in \Z$
  \begin{equation*}
    \sum_{j= -\ell}^\infty |\det A|^{\alpha jq} \Big[ \big( \DoubleStar{j}
    f (x) \big)^{q/r} \Big]^r \lesssim \sum_{j= -\ell}^\infty |\det
    A|^{\alpha j q}\Bigg[ |\det A|^j \int_{\R^d} \bigg( \frac{|(f
      \ast\varphi_j)(y)|} {\nu_{\beta}\big(A^j(x-y)\big)} \bigg)^{q/r}
     dy \Bigg]^r.
  \end{equation*}
  For fixed, but arbitrary $x \in \R^d$ and $\ell \in \Z$, we partition
  $
  \R^d = Q_\ell(x) \cup \bigcup_{k=0}^\infty \mathring{Q}_{\ell+k+1} (x)$,
  where
  \[ Q_\ell(x):= A^\ell \Omega + x \quad \text{and} \quad \mathring{Q}_{\ell+k+1} (x) := Q_{\ell+k+1} (x) \setminus Q_{\ell+k} (x).
  \]
  Combining this
  with the simple fact that $(a+b)^r \lesssim a^r + b^r$ for
  $a,b \geq 0$ yields
  \begin{align*}
   &\sum_{j= -\ell}^\infty |\det A|^{\alpha jq} \Big[ \big( \DoubleStar{j}
    f (x) \big)^{q/r} \Big]^r \\
    & \quad \quad \lesssim  \sum_{j= -\ell}^\infty  |\det A|^{\alpha j q}\Big[
      |\det A|^j \int_{Q_\ell (x)} \bigg( \frac{|(f \ast\varphi_j)(y)|}
      {\nu_{\beta}\big(A^j(x-y)\big)} \bigg)^{q/r}
       dy \Big]^r  \\
    & \quad \quad \quad \quad + \sum_{j= -\ell}^\infty  |\det A|^{\alpha j q}\Big[
      |\det A|^j \sum_{k=0}^\infty \int_{\mathring{Q}_{\ell+k+1} (x)}
      \bigg( \frac{|(f \ast\varphi_j)(y)|}
      {\nu_{\beta}\big(A^j(x-y)\big)} \bigg)^{q/r}
       dy \Big]^r \\
     & \quad \quad =: S_1 + S_2.
  \end{align*}
 In the remainder, the series defining $S_1$ and $S_2$ will be estimated.
 \\~\\
  \textit{Step 3.1.} We look at the sum $S_1$ first.  Note that since
  $|\det A| > 1$, there exists $M \in \N$ such that
  $|\det A|^M \geq 2 C$, where $C>0$ denotes the constant in the
  triangle-inequality for $\rho_A$ (cf. \Cref{lem:QuasiNorm}).
  A straightforward computation
  shows that $Q_\ell (x) \subset Q_{\ell+M}(w)$ for all $w \in \R^d$ and
  $\ell \in \Z$ whenever $x \in Q_\ell(w)$.
  Therefore, for all $w \in \R^d$, $\ell \in \Z$ and
  $x \in Q_\ell(w)$, it follows that
  \begin{align*}
    &|\det A|^j \int_{Q_\ell (x)}
     \bigg( \frac{|(f \ast\varphi_j)(y)|}
    {\nu_{\beta}\big(A^j(x-y)\big)}\bigg)^{q/r}  dy \\
    & \quad \quad \leq |\det A|^j \int_{\R^d} \bigg( \frac{|(f \ast\varphi_j)(y)|}
      {\nu_{\beta}\big(A^j(x-y)\big)} \bigg)^{q/r} \Indicator_{Q_{\ell+M}(w)}(y)  \, dy \\
    & \quad \quad = |\det A|^j \int_{Q_{-j}(x)} \bigg( \frac{|(f \ast\varphi_j)(y)|}
      {\nu_{\beta}\big(A^j(x-y)\big)} \bigg)^{q/r} \Indicator_{Q_{\ell+M}(w)}(y)  \, dy\\
    & \quad \quad \quad \quad + \sum_{m=0}^\infty |\det A|^j \int_{\mathring{Q}_{m-j+1}(x)}
     \bigg( \frac{|(f \ast\varphi_j)(y)|}
      {\nu_{\beta}\big(A^j(x-y)\big)} \bigg)^{q/r} \Indicator_{Q_{\ell+M}(w)} (y) \, dy \\
    &\quad \quad =: I_1 + I_2.
  \end{align*}
  To estimate the terms $I_1$ and $I_2$, we will use the (anisotropic)
  Hardy-Littlewood maximal operator for locally integrable
  $f : \mathbb{R}^d \to \mathbb{C}$ given by
  \[
  M_{\rho_A} f(x) = \sup_{B \ni x} \;
      \dashint_B
        |f(y)|
      \; dy, \quad x \in \mathbb{R}^d,
  \]
  where $B = B_{\rho_A}(z, s)$ ranges over all metric balls containing
  $x$.

  For estimating $I_1$, note that
  $\nu_{\beta}\big(A^j(x-y)\big)^{-q/r}=\big(1+\rho_A(A^j(x-y))\big)^{-\beta q/r} \leq 1$ yields
  \begin{align*}
    I_1
    & \leq |\det A|^j \int_{Q_{-j}(x)} |(f \ast\varphi_j)(y)|^{q/r}
      \Indicator_{Q_{\ell+M}(w)}(y)  \, dy \\
     & \leq M_{\rho_A} \big(|f \ast\varphi_j|^{q/r} \Indicator_{Q_{\ell+M}(w)} \big)(x).
      \numberthis \label{eq:norm-equiv-Step3-I1}
  \end{align*}

  For estimating $I_2$, note that $\rho_A(A^j (x-y)) = |\det A|^{m}$ for
  $y \in \mathring{Q}_{m-j+1}(x)$ by definition of $\rho_A$ (see
  \eqref{eq:step_norm}). This and setting
  $\delta:= \beta q/r - 1$ implies
  \begin{align*}
    I_2
    & \leq \sum_{m=0}^\infty |\det A|^{j - \beta q m /r} \int_{\mathring{Q}_{m-j+1}(x)}
      |(f \ast\varphi_j)(y)|^{q/r} \Indicator_{Q_{\ell+M}(w)} (y) \, dy \\
    & \leq \sum_{m=0}^\infty |\det A|^{-m \delta+1}|\det A|^{j -  m - 1} \int_{Q_{m-j+1}(x)}
      |(f \ast\varphi_j)(y)|^{q/r} \Indicator_{Q_{\ell+M}(w)} (y) \, dy \\
    & \leq  M_{\rho_A}\big(|f \ast\varphi_j|^{q/r}  \Indicator_{Q_{\ell+M}(w)} \big) (x)
      \sum_{m=0}^\infty |\det A|^{-m \delta+1} \\
    &\lesssim  M_{\rho_A}\big(|f \ast\varphi_j|^{q/r}  \Indicator_{Q_{\ell+M}(w)} \big) (x),
      \numberthis \label{eq:norm-equiv-Step3-I2}
  \end{align*}
  where the last inequality used that $\delta= \beta q/r - 1 >0$.  Here,
  the implicit constant only depends on $A$, $q$ and $\beta$.

  Combining \eqref{eq:norm-equiv-Step3-I1} and
  \eqref{eq:norm-equiv-Step3-I2} shows for $x \in Q_\ell(w)$ that
  \begin{align*}
    S_1
    \lesssim \sum_{j = - \ell}^\infty |\det A|^{\alpha j q}
    \Big[
    M_{\rho_A}\big(|f \ast\varphi_j|^{q/r}  \Indicator_{Q_{\ell+M}(w)} \big) (x)
    \Big]^r.
  \end{align*}
  Thus, averaging over $x \in Q_\ell(w)$ and applying the
  maximal inequalities for $L^r(\R^d)$ (see, e.g., \cite[Theorem 1.2]{grafakos2009vector}), yield
  \begin{align*}
    \dashint_{Q_\ell(w)} \! \! \! \! S_1 \, dx
     & \lesssim \; \sum_{j = - \ell}^\infty |\det A|^{\alpha j q} \;
    \dashint_{Q_\ell(w)}\Big[
    M_{\rho_A}\big(|f \ast\varphi_j|^{q/r}  \Indicator_{Q_{\ell+M}(w)} \big) (x)
       \Big]^r   dx \\
     & \lesssim \; \sum_{j = - (\ell+M)}^\infty \! \! \! \! \! \! |\det A|^{\alpha j q}
    \frac{1}{\Lebesgue{Q_{\ell+M}(w)}} \int_{\R^d}\Big[
    M_{\rho_A}\big(|f \ast\varphi_j|^{q/r}  \Indicator_{Q_{\ell+M}(w)} \big) (x)
       \Big]^r   dx \\
    & \lesssim \; \sum_{j = - (\ell+M)}^\infty \! \! \! \! \! \!|\det A|^{\alpha j q}
    \frac{1}{\Lebesgue{Q_{\ell+M}(w)}} \int_{\R^d}
    |(f \ast\varphi_j)(x)|^{q}  \Indicator_{Q_{\ell+M}(w)}(x)
       \, dx.
  \end{align*}
  Lastly, taking the suprema over $w \in \R^d$ and $\ell \in \Z$ yields
  \begin{equation}
    \label{eq:norm-equiv-Step3-S1}
    \sup_{\ell \in \Z, w \in \R^d} \;
    \dashint_{Q_\ell(w)} \! \! \! \! S_1 \, dx
    \lesssim \|f\|_{\TLi{}}^q,
  \end{equation}
  with implicit constant depending on $A,\varphi,q$, and $\beta$.
  \\~\\
  \textit{Step 3.2.} In this step, we deal with the sum $S_2$.  Recall again that, for $y \in \mathring{Q}_{\ell+k+1} (x)$,
  $\rho_A((A^j(x-y))) = |\det A|^{j+k+\ell}$. Hence,
  \begin{align*}
    |\det A|^j
    &\sum_{k=0}^\infty \int_{\mathring{Q}_{\ell+k+1} (x)}
      \bigg( \frac{|(f \ast\varphi_j)(y)|}
      {\nu_{\beta}\big(A^j(x-y)\big)} \bigg)^{q/r}
       dy \\
    & \leq |\det A|^j \sum_{k=0}^\infty
      |\det A|^{-(j+k+\ell)\beta q /r}
      \int_{Q_{\ell+k+1} (x)}
      |(f \ast\varphi_j)(y)|^{q/r}
      \, dy \\
    & = |\det A|^{-\delta(j+\ell)+1} \sum_{k=0}^\infty
      |\det A|^{-\delta k}
      \dashint_{Q_{\ell+k+1} (x)}
      |(f \ast\varphi_j)(y)|^{q/r}
      \, dy ,
  \end{align*}
  where again $\delta:= \beta q / r -1 >0$.  Note that
  $|\det A|^{-\delta(j+\ell)} \leq 1$ for $j \geq - \ell$, which implies that
  \begin{align*}
    S_2
    & \lesssim \sum_{j = - \ell}^\infty |\det A|^{\alpha j q}
      \Big[\sum_{k=0}^\infty
      |\det A|^{-\delta k}
      \dashint_{Q_{\ell+k+1} (x)}
      |(f \ast\varphi_j)(y)|^{q/r}
      \, dy
      \Big]^r \\
    & \lesssim  \Big[\sum_{k=0}^\infty |\det A|^{-\delta k}
      \Big[ \sum_{j = - \ell}^\infty \Big(|\det A|^{\alpha j q/r}
      \dashint_{Q_{\ell+k+1} (x)}
      |(f \ast\varphi_j)(y)|^{q/r}
      \, dy \Big)^r
      \Big]^{1/r}
      \Big]^r,
  \end{align*}
  where we used Minkowski's integral inequality (see, e.g., \cite[Appendix 1]{stein1970singular}) to obtain the last line.
  An application of Jensen's inequality to the integral yields
  \begin{align*}
    \Big(\dashint_{Q_{\ell+k+1} (x)}
      |(f \ast\varphi_j)(y)|^{q/r}
    \, dy \Big)^r
    \leq
    \dashint_{Q_{\ell+k+1} (x)}
      |(f \ast\varphi_j)(y)|^{q}
    \, dy ,
  \end{align*}
  and consequently
  \begin{align*}
    S_2
    & \lesssim  \Big[\sum_{k=0}^\infty |\det A|^{-\delta k}
      \Big[ \sum_{j = - \ell}^\infty |\det A|^{\alpha j q}
      \dashint_{Q_{\ell+k+1} (x)}
      |(f \ast\varphi_j)(y)|^{q}
      \, dy
      \Big]^{1/r}
      \Big]^r \\
    & \leq  \Big[\sum_{k=0}^\infty |\det A|^{-\delta k}
      \Big[ \sum_{j = -( \ell+k+1)}^\infty |\det A|^{\alpha j q}
      \dashint_{Q_{\ell+k+1} (x)}
      |(f \ast\varphi_j)(y)|^{q}
      \, dy
      \Big]^{1/r}
      \Big]^r\\
    & \leq \sup_{\ell' \in \Z, x \in \R^d}
      \bigg ( \sum_{j = - \ell'}^\infty |\det A|^{\alpha j q}
      \dashint_{Q_{\ell'} (x)}
      |(f \ast\varphi_j)(y)|^{q}
      \, dy \bigg )
      \Big[\sum_{k=0}^\infty |\det A|^{-\delta k} \Big]^r \\
      &\lesssim \sup_{\ell' \in \Z, x \in \R^d}
      \bigg ( \sum_{j = - \ell'}^\infty |\det A|^{\alpha j q}
      \dashint_{Q_{\ell'} (x)}
      |(f \ast\varphi_j)(y)|^{q}
      \, dy \bigg ),
      \numberthis \label{eq:norm-equiv-Step3-eq2}
  \end{align*}
  where we used the index shift $\ell' = \ell + k +1$ in the penultimate
  estimate. Since the implicit constants are independent of
  $w \in \R^d$ and $\ell \in \Z$, it follows that
  \begin{equation}
    \label{eq:norm-equiv-Step3-S2}
    \sup_{\ell \in \Z, w \in \R^d}
    \;
    \dashint_{Q_\ell(w)} S_2 \, dx
    \lesssim \|f\|_{\TLi{}}^q .
  \end{equation}

  Overall, combining the estimates \eqref{eq:norm-equiv-Step3-S1} and
  \eqref{eq:norm-equiv-Step3-S2} finishes the proof.
\end{proof}

\section{The case $p=q=\infty$} \label{sec:pqinfty}
Let $A \in \mathrm{GL}(d, \mathbb{R})$ be expansive and suppose that
$\varphi \in \Schwartz(\R^d)$ has compact Fourier support and
satisfies \eqref{eq:analyzing_support} and
\eqref{eq:analyzing_positive}.
Following \cite{bownik2005atomic}, for fixed $\alpha \in \mathbb{R}$, the
associated \emph{homogeneous anisotropic Besov space}
$\Bii = \Bii(A,\varphi)$ is defined as the set of all $f \in \SP$ for
which
\[
  \| f \|_{\Bii} := \sup_{j \in \Z} \, \sup_{x \in \R^d}
  |\det A|^{\alpha j} |(f \ast \varphi_j)(x)|  < \infty.
\]
The space $\Bii$ is continuously embedded into $\SP$ and complete with respect to the quasi-norm $\| \cdot \|_{\Bii}$, cf. \cite[Proposition 3.3]{bownik2005atomic}. In addition,
it is independent of the choice of defining vector $\varphi \in \mathcal{S}(\R^d)$ by \cite[Corollary 3.7]{bownik2005atomic}.

The following theorem relates the anisotropic Triebel-Lizorkin spaces
$\TLi$, with $0 < q \leq \infty$, to the anisotropic Besov space
$\Bii$.  In particular, it shows that $\Bii = \TLii$, providing a characterization
of $\Bii$ in terms of Peetre-type maximal functions.
It also shows that the definition of $\TLii$ given in \Cref{sub:TLSpaces} coincides with the definition in \cite[Section 3]{bownik2007anisotropic}.

\begin{theorem} \label{thm:pqinfty}
  Let $A \in \mathrm{GL}(d, \mathbb{R})$ be expansive.  Assume that $\varphi \in \Schwartz(\R^d)$ has compact
  Fourier support and satisfies \eqref{eq:analyzing_support} and
  \eqref{eq:analyzing_positive}.

  Then, for all $q \in (0, \infty]$ and
  $\alpha \in \mathbb{R}$, we have the continuous embedding
  \begin{equation*}
     \TLi(A,\varphi) \subset \TLii(A,\varphi).
  \end{equation*}
  In addition, $\TLii(A,\varphi) = \Bii(A,\varphi)$.
\end{theorem}

\begin{proof}
  The inequality $\|f\|_{\TLii} \leq \|f\|_{\Bii}$ is immediate. To show that $\| f \|_{\Bii} \lesssim \| f \|_{\TLii}$,
  the norm equivalences of Theorem~\ref{thm:norm_equiv}
  will be used; see also \Cref{rem:general_expansive}.  For this, suppose
  $\beta >1$. Then, for all $\ell \in \Z$, $w \in \R^d$, we see that
  \begin{align*}
    \dashint_{Q_\ell(w)} |\det A|^{ - \alpha \ell}
    \DoubleStar{- \ell} f(x) \, dx
    & \leq \sup_{\ell \in \Z, w \in \R^d} \sup_{j \in \Z, j \geq -
      \ell} \bigg(
      \dashint_{Q_\ell(w)} |\det A|^{\alpha j}
      \DoubleStar{j} f(x) \, dx \bigg) \\
    & \lesssim \|f\|_{\TLii},
      \numberthis \label{eq:case-p=q=infty-eq1}
  \end{align*}
In particular, the inequality~\eqref{eq:case-p=q=infty-eq1}
  implies that, for every $\ell \in \Z$ and
  $w \in \R^d$ there exists $x_w \in Q_\ell(w)$ such that
  \begin{equation}
    \label{eq:case-p=q=infty-eq2}
    |\det A|^{ - \alpha \ell}
    \DoubleStar{- \ell} f(x_w) \lesssim \|f\|_{\TLii}.
  \end{equation}
  If $z \in Q_\ell(w)$, then
  $\rho_A\big(A^{-\ell}(z-x_w)\big) \lesssim
  \rho_A\big(A^{-\ell}(z-w)\big) +\rho_A\big(A^{-\ell}(w-x_w)\big)
  \lesssim 1$ with implicit constant depending only on $A$.  Hence, for
  all $z \in Q_\ell(w)$, it holds that
  \begin{equation}
    \label{eq:case-p=q=infty-eq3}
    \DoubleStar{- \ell} f(x_w) \geq  \frac{|(f \ast \varphi_{- \ell})(z)|}
    {\big( 1+ \rho_A\big(A^{-\ell}(z-x_w)\big)\big)^\beta} \gtrsim_{A, \beta} |(f \ast \varphi_{- \ell})(z)|.
  \end{equation}
  Since
  $\bigcup_{w \in \R^d} Q_\ell(w) = \R^d$ for fixed, but arbitrary,
  $\ell \in \Z$, we see by combining \eqref{eq:case-p=q=infty-eq2} and
  \eqref{eq:case-p=q=infty-eq3} that
  \begin{equation*}
   |\det A|^{ - \alpha \ell}  |(f \ast \varphi_{- \ell})(z)| \lesssim  \|f\|_{\TLii}
  \end{equation*}
  for all $z \in \R^d$ and $\ell \in \Z$, which shows
  $\TLii(A,\varphi) = \Bii(A,\varphi)$ with equivalent norms.

  An analogous argument using \eqref{eq:norm_equiv} gives
  $\|f\|_{\Bii} \lesssim \|f\|_{\TLi}$, which completes the proof.
\end{proof}

\Cref{thm:pqinfty} is
an extension of \cite[Theorem 3]{bui2000characterization} to the
anisotropic setting.

\section{Wavelet coefficient decay and Peetre-type spaces}
Throughout this section, let  $A \in \mathrm{GL}(d, \mathbb{R})$ be an exponential matrix. Define
the associated semi-direct product
$
  G_A = \mathbb{R}^d \rtimes_A \mathbb{R}
      = \{
          (x, s)
          :
          x \in \mathbb{R}^d, s \in \mathbb{R}
        \}
$
with group operations
\begin{align} \label{eq:SemidirectProductDefinition}
(x,s) (y,t) = (x + A^sy, s+t) \quad \text{and} \quad
    (x,s)^{-1} = (-A^{-s} x, -s).
\end{align}
Left Haar measure on $G_A$ is given
by $d \mu_{G_A} (x,s) = |\det A|^{-s} ds \, dx$ and the modular
function on $G_A$ is $\Delta_{G_A} (x,s) = |\det A|^{-s}$.  To ease notation,
we will often write $\mu := \mu_{G_A}$.

\subsection{Wavelet transforms}
The group $G_A = \mathbb{R}^d \rtimes_A \mathbb{R}$ acts unitarily on $L^2 (\mathbb{R}^d)$ by means of the \emph{quasi-regular representation} $\pi$, defined by
\[
\pi(x,s) f = |\det A|^{-s/2} f(A^{-s} (\cdot - x)), \quad (x,s) \in G_A.
\]
For a fixed vector $\psi \in L^2 (\mathbb{R}^d) \setminus \{0\}$, the associated wavelet transform $W_{\psi} : L^2 (\mathbb{R}^d) \to L^{\infty} (G_A)$ is defined through the coefficients
\[
W_{\psi} f (x,s) = \langle f, \pi(x,s) \psi \rangle, \quad f \in L^2 (\mathbb{R}^d).
\]
The vector $\psi$ is called \emph{admissible} if $W_{\psi} : L^2 (\mathbb{R}^d) \to L^2 (G_A)$ is an isometry.

The following lemma is a special case of \cite[Theorem 1]{fuehr2002continuous} and \cite[Theorem 1.1]{laugesen2002characterization}.

\begin{lemma}[\cite{laugesen2002characterization, fuehr2002continuous}]
  Let $\psi \in L^2 (\mathbb{R}^d)$. Then $\psi$ is admissible if and only if
  \[
  \int_{\mathbb{R}} | \widehat{\psi} ((A^*)^s \xi) |^2 \; ds = 1
  \]
  for a.e. $\xi \in \mathbb{R}^d$.
\end{lemma}

The significance of an admissible vector $\psi$ is that $W_{\psi}^* W_{\psi} = I_{L^2 (\mathbb{R}^d)}$, and hence that
the weak-sense integral formula
\[
f = \int_{G_A} W_{\psi} f (g) \pi(g) \psi \; d\mu_{G_A} (g)
\]
holds for every $f \in L^2 (\mathbb{R}^d)$.
This, combined with fact that $W_{\psi} : L^2 (\mathbb{R}^d) \to L^2 (G_A)$ intertwines the action of $\pi$ and left translation $L_h F = F(h^{-1} \cdot)$ on $L^2 (G_A)$, also yields that
\begin{align} \label{eq:reproL2}
W_{\varphi} f = W_{\psi} f \ast W_{\varphi} \psi,
\end{align}
for all $f, \varphi \in L^2 (\mathbb{R}^d)$.
The identity \eqref{eq:reproL2} will be referred to as a \emph{reproducing formula}.

Henceforth, it will always be assumed that $A \in \mathrm{GL}(d, \mathbb{R})$ is both exponential and expansive. This is essential for the existence of admissible vectors $\psi \in \mathcal{S} (\mathbb{R}^d)$ with compact Fourier support, as the following result shows.

\begin{lemma}[\cite{currey2016integrable, groechenig1991describing, schulz2004projections, bownik2003anisotropic}] \label{lem:admissible_expansive}
  Let $A \in \mathrm{GL}(d, \mathbb{R})$ be an exponential matrix. The following assertions are equivalent:
  \begin{enumerate}
      \item[(i)] The matrix $A$ or its inverse $A^{-1}$ is expansive.
      \item[(ii)] There exists an admissible $\psi \in \mathcal{S} (\mathbb{R}^d)$ with $\widehat{\psi} \in C_c^{\infty} (\mathbb{R}^d)$.
  \end{enumerate}
  In addition, if $A$ is an expansive matrix, then the admissible
  vector $\psi \in L^2 (\mathbb{R}^d)$ can be chosen such that
  $\widehat{\psi} \in C_c^{\infty} (\mathbb{R}^d \setminus \{0\})$
   satisfies condition~\eqref{eq:analyzing_positive}.
\end{lemma}

We will also need to use wavelet transforms of distributions. For this, consider the subspace of $\mathcal{S}(\mathbb{R}^d)$ given by
  \[
  \mathcal{S}_0 (\mathbb{R}^d) := \bigg\{ \varphi \in \mathcal{S} (\mathbb{R}^d) : \int_{\mathbb{R}^d} \varphi (x) x^{\alpha} \; dx = 0, \quad \forall \alpha \in \mathbb{N}_0^d \bigg\},
  \]
and equip it with the subspace topology of $\mathcal{S}(\mathbb{R}^d)$. Its topological dual space will be denoted by $\mathcal{S}_0' (\mathbb{R}^d)$ and will often be identified with $\mathcal{S}'(\mathbb{R}^d) / \mathcal{P} (\mathbb{R}^d)$, see, e.g., \cite[Proposition 1.1.3]{grafakos2014modern}.
The dual bracket between $\mathcal{S}_0$ and $\mathcal{S}_0' (\mathbb{R}^d)$ is denoted by
\[
\langle \cdot , \cdot \rangle : \mathcal{S}_0' (\mathbb{R}^d) \times \mathcal{S}_0 (\mathbb{R}^d), \quad \langle f, h \rangle := f(\overline{h}).
\]
Note that this pairing is conjugate-linear in the second variable.

For a fixed $\psi \in \mathcal{S}_0 (\mathbb{R}^d) \setminus \{0\}$, the extended wavelet transform of $f \in \mathcal{S}'_0 (\mathbb{R}^d)$ is defined as
\[
W_{\psi} f (x,s) = \langle f, \pi(x,s) \psi \rangle, \quad (x,s) \in G_A.
\]
By the continuity of the map $(x,s) \mapsto \pi(x,s) \varphi$ from $\mathbb{R}^d \times \mathbb{R}$ into $\mathcal{S}(\mathbb{R}^d)$, the transform $W_{\psi} : \mathcal{S}_0' (\mathbb{R}^d) \to C(G_A)$ is well-defined.

The reproducing formula \eqref{eq:reproL2} can be naturally extended to $\mathcal{S}'_0 (\mathbb{R}^d)$. See  \cite[Lemma 4.7 and Lemma 4.8]{KvVV2021anisotropic} for a proof of the following result.

\begin{lemma}[\cite{KvVV2021anisotropic}]
  Let $\psi \in \mathcal{S}_0 (\mathbb{R}^d)$ be an admissible vector. Then, for arbitrary $f \in \mathcal{S}'_0 (\mathbb{R}^d)$ and $\varphi \in \mathcal{S}_0 (\mathbb{R}^d)$,
  \[
  \langle f, \varphi \rangle = \int_{G_A} W_{\psi} f(g) \overline{W_{\psi} \varphi (g) } \; d\mu_{G_A} (g).
  \]
  In particular, the identity $W_{\varphi} f = W_{\psi} f \ast W_{\varphi} \psi$ holds.
\end{lemma}

\subsection{Peetre-type spaces}
As in \cite{KvVV2021anisotropic}, we define an auxiliary class of Peetre-type spaces $\PTi$ on the semi-direct product $G_A = \mathbb{R}^d \rtimes \mathbb{R}$. These spaces are an essential ingredient for identifying Triebel-Lizorkin spaces with associated coorbit spaces \cite{velthoven2022quasi, feichtinger1989banach}.

In contrast to the spaces $\PT$ defined in \cite[Definition 5.1]{KvVV2021anisotropic} for $p < \infty$, the spaces $\PTi$ will only be defined through averages over small scales.

\begin{definition} \label{def:PTi}
  Let $A \in \mathrm{GL}(d, \mathbb{R})$ be expansive and exponential and let $\Omega = \Omega_A$ be an associated ellipsoid as provided by \Cref{lem:QuasiNorm}. For $\ell \in \Z$ and $w \in \R^d$, set $Q_{\ell} (w) := A^{\ell} \Omega + w$.

  For $\alpha \in \mathbb{R}, \beta > 0$, and $q \in (0, \infty)$, the
  Peetre-type space $\PTi (G_A)$  is
  defined as the space of all (equivalence classes of a.e.~equal)
  measurable $F : G_A \to \mathbb{C}$ such that
  \[
    \| F \|_{\PTi} := \sup_{\ell \in \Z, w \in \R^d}  \bigg(
    \dashint_{Q_{\ell} (w)} \int_{-\infty}^\ell \bigg[ |\det
    A|^{\alpha s} \esssup_{z \in \mathbb{R}^d} \frac{|F(x+z,s)|} {(1 +
      \rho_A (A^{-s} z))^{\beta}} \bigg]^q \frac{ds \, dx}{|\det A|^{s}}
    \bigg)^{1/q} < \infty.
  \]
  For $q = \infty$, the space $\PTii (G_A)$ consists of all measurable
  $F : G_A \to \mathbb{C}$ satisfying
  \[
    \| F \|_{\PTii} := \sup_{\ell \in \Z, w \in \R^d} \sup_{s \in \R,
      s \leq \ell} \bigg( \dashint_{Q_{\ell} (w)} \bigg[ |\det
    A|^{\alpha s} \esssup_{z \in \mathbb{R}^d} \frac{|F(x+z,s)|} {(1 +
      \rho_A (A^{-s} z))^{\beta}} \bigg] \, dx \bigg) < \infty.
  \]
\end{definition}

The following lemma collects some basic properties of $\PTi(G_A)$ and
gives explicit estimates for the operator norm of left and right
translation. The estimates involve the following weight function
\[
v : G_A \to [1, \infty), \quad v(y,t) = \sup_{(z, u) \in G_A} \frac{1 + \rho_A (A^{-u} z)}{1 +
    \rho_A (A^{-u} A^t z - y)}.
\]
The function $v$ is measurable and submultiplicative by \cite[Lemma 5.2]{KvVV2021anisotropic}.

\begin{lemma}\label{lem:TranslationNormBounds}
  Let $\alpha \in \mathbb{R}, \beta > 0$, and $q \in (0, \infty]$.
  Then the Peetre-type space $\PTi(G_A)$ is a solid quasi-Banach
  function space (Banach function space if $q \geq 1$).  Moreover, the space $\PTi(G_A)$ is left- and right-translation invariant and there exists $N = N(A, \Omega) \in \mathbb{N}$ such that,
  for all $(y, t) \in G_A$, the operator norms can be bounded by
  \begin{align*}
    \vertiii{L_{(y,t)}}
    \leq |\det A|^{ t (\alpha-1/q) + (N+1)/q },\quad
    \vertiii{R_{(y,t)}}
    \leq \begin{cases}
      |\det A|^{-t(\alpha-2/q) + 1/q} (v(y,t))^{\beta} & \text{if } t > 0 ,  \\[0.1cm]
      |\det A|^{-t(\alpha-1/q)}  (v(y,t))^{\beta} & \text{if } t \leq 0,
       \end{cases}
  \end{align*}
  if $q < \infty$, and
  \begin{align*}
    \vertiii{L_{(y,t)}}
    \leq |\det A|^{ t \alpha + N+1 },
    \quad
    \vertiii{R_{(y,t)}}
    \leq \begin{cases}
        |\det A|^{-t(\alpha-1) + 1}  (v(y,t))^{\beta}      & \text{if } t > 0 ,  \\[0.1cm]
        |\det A|^{- t \alpha } (v(y,t))^{\beta} & \text{if } t \leq 0,
       \end{cases}
  \end{align*}
  otherwise. Here, the operator norm is written $\vertiii{\cdot} :=\| \cdot \|_{\PTi \to \PTi}$.
\end{lemma}

\begin{proof}
  Most of the quasi-norm properties for $\| \cdot \|_{\PTi}$ can be
  easily verified from the definition, while the positive definiteness follows from
  \cite[Lemma B.1]{KvVV2021anisotropic}.  It is clearl that $\| \cdot \|_{\PTi}$ is solid and for $q \geq 1$ a norm.
  The completeness follows from $\| \cdot \|_{\PTi}$ satisfying the \emph{Fatou
    property} (see \cite[Section~65, Theorem~1]{zaanen1967integration}
  and \cite[Lemma~2.2.15]{VoigtlaenderPhDThesis}), which is easily
  verified by a straightforward computation using Fatou's lemma, see, e.g., the proof of \cite[Lemma 5.3]{KvVV2021anisotropic}.

  It remains to prove the translation invariance and associated norm estimates.  We will only
  consider $q \in (0, \infty)$, the arguments for $q=\infty$ are
  analogous.  To this end, let $F \in \PTi(G_A)$ and
  ${(y,t) \in \mathbb{R}^d \times \mathbb{R}}$ be arbitrary.  Then the
  substitutions $\widetilde{x} = A^{-t}(x - y)$,
  $\widetilde{w} = A^{-t}(w-y)$, and $\widetilde{z} = A^{-t} z$, as
  well as $\widetilde{s} = s-t$, yield
  \begin{align*}
    \| &L_{(y,t)}  F \|_{\PTi}^q \\
       &= \sup_{\ell \in \Z, w \in \R^d} \bigg(
         \dashint_{Q_\ell(w)} \int_{-\infty}^\ell \bigg[ |\det A|^{\alpha s} \esssup_{z \in \mathbb{R}^d}
         \frac{|F(A^{-t}(x+z-y),s-t)|}
         {(1 + \rho_A (A^{-s} z))^{\beta}}
         \bigg]^q
         \frac{ds \, dx}{|\det A|^{s}}
         \bigg) \\
       &= \sup_{\ell \in \Z, \widetilde{w} \in \R^d} \bigg(
         |\det A|^{-\ell}\int_{Q_{\ell-t}( \widetilde{w})}
         \int_{-\infty}^{\ell} \bigg[ |\det A|^{\alpha s} \esssup_{\widetilde{z} \in \mathbb{R}^d}
         \frac{|F(\widetilde{x}+\widetilde{z}, s-t)|}
         {(1 + \rho_A (A^{t-s} \widetilde{z}))^{\beta}}
         \bigg]^q
         \frac{ds \, d\widetilde{x}}{|\det A|^{s-t}}
         \bigg) \\
       &= \sup_{\ell \in \Z, \widetilde{w} \in \R^d} \bigg(
         |\det A|^{-\ell}\int_{Q_{\ell-t}( \widetilde{w})}
         \int_{-\infty}^{\ell-t} \bigg[ |\det A|^{\alpha (\widetilde{s} +t)} \esssup_{\widetilde{z} \in \mathbb{R}^d}
         \frac{|F(\widetilde{x}+\widetilde{z}, \widetilde{s})|}
         {(1 + \rho_A (A^{-\widetilde{s}} \widetilde{z}))^{\beta}}
         \bigg]^q
         \frac{d\widetilde{s} \, d\widetilde{x}}{|\det A|^{\widetilde{s}}}
         \bigg),
  \end{align*}
  where the transformation of the domain is
  $A^{-t}(Q_\ell(w)-y) = A^{-t}(A^\ell\Omega +w -y) = A^{\ell-t}\Omega
  + \widetilde{w}$.
  To estimate this further, we
  decompose $t = k + t'$ with $k \in \Z$ and $t' \in [0,1)$.  By
  \cite[Lemma 2.4]{KvVV2021anisotropic}, there exists
  $N = N(A, \Omega) \in \N$ such that $A^{-t'}\Omega \subset A^N \Omega$ for all
  $t' \in [0,1)$, and hence
  $Q_{\ell-t}( \widetilde{w}) \subset Q_{\ell-k+N}( \widetilde{w})$.
  Increasing the upper limit of the inner integral from $\ell -t$ to $\ell -k +N$
  and substituting $\widetilde{\ell} = \ell -k +N \in \Z$ gives
  \begin{align*}
     \| &L_{(y,t)}  F \|_{\PTi}^q \\
    &\leq \sup_{\widetilde{\ell} \in \Z, \widetilde{w} \in \R^d} \bigg(
      |\det A|^{-\widetilde{\ell} -k+N}\int_{Q_{\widetilde{\ell}}( \widetilde{w})}
      \int_{-\infty}^{\widetilde{\ell}}
      \bigg[ |\det A|^{\alpha (\widetilde{s} +t)}
      \esssup_{\widetilde{z} \in \mathbb{R}^d}
      \frac{|F(\widetilde{x}+\widetilde{z}, \widetilde{s})|}
      {(1 + \rho_A (A^{-\widetilde{s}} \widetilde{z}))^{\beta}}
      \bigg]^q
      \frac{d\widetilde{s} \, d\widetilde{x}}{|\det A|^{\widetilde{s}}}
      \bigg) \\
    &\leq \sup_{\widetilde{\ell} \in \Z, \widetilde{w} \in \R^d} \bigg(
      |\det A|^{-t+N+1}\dashint_{Q_{\widetilde{\ell}}( \widetilde{w})}
      \int_{-\infty}^{\widetilde{\ell}}
      \bigg[ |\det A|^{\alpha (\widetilde{s} +t)}
      \esssup_{\widetilde{z} \in \mathbb{R}^d}
      \frac{|F(\widetilde{x}+\widetilde{z}, \widetilde{s})|}
      {(1 + \rho_A (A^{-\widetilde{s}} \widetilde{z}))^{\beta}}
      \bigg]^q
      \frac{d\widetilde{s}  \, d\widetilde{x}}{|\det A|^{\widetilde{s}}}
      \bigg) \\
    &= |\det A|^{t(\alpha q -1)+N+1} \|F\|_{\PTi}^q.
  \end{align*}

  For the right-translation, a direct calculation using the substitutions $\widetilde{z} = z + A^s y$
  and $\widetilde{s} = s + t$ shows that
  \begin{align*}
    \| &R_{(y,t)} F \|_{\PTi}^q \\
       &= \sup_{\ell \in \Z, w \in \R^d} \bigg(
         \dashint_{Q_\ell(w)} \int_{-\infty}^\ell \bigg[
         |\det A|^{\alpha s}
         \esssup_{z \in \mathbb{R}^d}
         \frac{|F(x+z+A^sy, s+t) |}
         {(1+\rho_A (A^{-s} z))^{\beta}}
         \bigg]^q
         \frac{ds \, dx}{|\det A|^{s}}
         \bigg) \\
       &= \sup_{\ell \in \Z, w \in \R^d} \bigg(
         \dashint_{Q_\ell(w)} \int_{-\infty}^{\ell+t} \bigg[
         |\det A|^{\alpha (\widetilde{s}-t)}
         \esssup_{\widetilde{z} \in \mathbb{R}^d}
         \frac{|F(x+\widetilde{z}, \widetilde{s}) |}
         {(1+\rho_A (A^{-\widetilde{s}} A^{t} \widetilde{z} - y))^{\beta}}
         \bigg]^q
         \frac{d\widetilde{s} \, dx}{|\det A|^{\widetilde{s}-t}}
         \bigg) \\
       &\leq |\det A|^{t - \alpha q t}  v(y,t)^{\beta q}  \\
       &\quad \quad \quad \quad \cdot \sup_{\ell \in \Z, w \in \R^d} \bigg(
         \dashint_{Q_\ell(w)} \int_{-\infty}^{\ell+t} \bigg[
         |\det A|^{\alpha \widetilde{s}}
         \esssup_{\widetilde{z} \in \mathbb{R}^d}
         \frac{|F(x+\widetilde{z}, \widetilde{s}) |}
         {(1+\rho_A (A^{-\widetilde{s}} \widetilde{z}))^{\beta}}
         \bigg]^q
         \frac{d\widetilde{s} \, dx}{|\det A|^{\widetilde{s}}}
         \bigg),
  \end{align*}
  where we used the fact that
  $\bigl(1+\rho_A (A^{-\widetilde{s}} A^{t} \widetilde{z} -
  y)\bigr)^{-1} \leq v(y,t)  (1 + \rho_A (A^{-\widetilde{s}}
  \widetilde{z}))^{-1}$ for all
  $(\widetilde{z},\widetilde{s}), (y,t) \in \R^d \times \R$ by
  definition.
  For $t \leq 0$, the claimed estimate follows immediately.  For $t > 0$, we
  again write $t = k + t'$ with $k \in \N_0$ and $t' \in [0, 1)$.  Then
  $\ell + t \leq \ell + k + 1$ and clearly
  $Q_\ell (w) = A^\ell \Omega + w \subset A^{\ell+k+1} \Omega + w =
  Q_{\ell+k+1}(w)$ by Lemma~\ref{lem:QuasiNorm}. Hence,

  \begin{align*}
    \sup_{\ell \in \Z, w \in \R^d}
    \mkern-18mu&\mkern18mu
                 \bigg(
                 \dashint_{Q_\ell(w)} \int_{-\infty}^{\ell+t} \bigg[
                 |\det A|^{\alpha \widetilde{s}}
                 \esssup_{\widetilde{z} \in \mathbb{R}^d}
                 \frac{|F(x+\widetilde{z}, \widetilde{s}) |}
                 {(1+\rho_A (A^{-\widetilde{s}} \widetilde{z}))^{\beta}}
                 \bigg]^q
                 \frac{d\widetilde{s} \,  dx
                 }{|\det A|^{\widetilde{s}}}
                 \bigg) \\
    \leq &\sup_{\ell \in \Z, w \in \R^d} \bigg(
           |\det A|^{-\ell} \int_{Q_{\ell+k+1} (w)} \int_{-\infty}^{\ell+k+1} \bigg[
           |\det A|^{\alpha \widetilde{s}}
           \esssup_{\widetilde{z} \in \mathbb{R}^d}
           \frac{|F(x+\widetilde{z}, \widetilde{s}) |}
           {(1+\rho_A (A^{-\widetilde{s}} \widetilde{z}))^{\beta}}
           \bigg]^q
           \frac{d\widetilde{s} \, dx}{|\det A|^{\widetilde{s}}}
           \bigg) \\
    = & |\det A|^{k+1} \sup_{\ell \in \Z, w \in \R^d} \bigg(
        \dashint_{Q_{\ell+k+1}(w)} \int_{-\infty}^{\ell+k+1} \bigg[
        |\det A|^{\alpha \widetilde{s}}
        \esssup_{\widetilde{z} \in \mathbb{R}^d}
        \frac{|F(x+\widetilde{z}, \widetilde{s}) |}
        {(1+\rho_A (A^{-\widetilde{s}} \widetilde{z}))^{\beta}}
        \bigg]^q
        \frac{d\widetilde{s} \, dx}{|\det A|^{\widetilde{s}}}
        \bigg) \\
    \leq & |\det A|^{t+1} \|F\|_{\PTi}^q,
  \end{align*}
  and consequently
  \begin{align*}
    \| R_{(y,t)} F \|_{\PTi}^q \leq |\det A|^{2t - \alpha q t +1}
    v(y,t)^{\beta q} \|F\|_{\PTi}^q,
  \end{align*}
  for $t>0$, which completes the proof.
\end{proof}

Lastly, we mention the following $r$-norm property of Peetre-type spaces.

\begin{lemma}
  \label{lem:r-norm-property}
  Let $\alpha \in \mathbb{R}$ and $\beta > 0$.  For
  $q \in (0, \infty]$, set $r := \min\{1, q \}$. Then
  $\| \cdot \|_{\PTi}$ is an $r$-norm, i.e.,
  \[\|F_1 + F_2 \|^r_{\PTi} \leq \| F_1 \|_{\PTi}^r + \| F_2 \|^r_{\PTi}  \quad \text{for} \quad F_1, F_2 \in \PTi.\]
\end{lemma}

\begin{proof}
  The claim follows immediately from the triangle inequality for
  $q \in [1, \infty]$, respectively a straightforward computation
  using the subadditivity of $x \mapsto |x|^q$ for $q \in (0, 1)$.
\end{proof}

\subsection{Standard envelope and control weight}

We recall the definition of a standard envelope given in
\cite[Definition~5.5]{KvVV2021anisotropic}.
\begin{definition}
  \label{def:standard-env}
  Let $\sigma = (\sigma_1, \sigma_2) \in (0,\infty)^2$ and let $L \in \R$.
  Then the \emph{standard envelope}
  $\Xi_{\sigma,L} : G_A \to (0,\infty)$ is given by
  $\Xi_{\sigma,L}(x,s):= \theta_\sigma (s) \cdot \eta_L(x,s)$, where
  \begin{equation*}
    \eta_L(x,s)
    := \big(
         1 + \min
         \{
           \rho_A (x), \rho_A(A^{-s} x)
         \}
       \big)^{-L}
    \qquad \text{and} \qquad
    \theta_\sigma (s)
    := \begin{cases}
        \sigma_1^s , & \text{if } s \geq 0, \\
        \sigma_2^s , & \text{if } s < 0.
       \end{cases}
  \end{equation*}
\end{definition}

A central notion in the theory of coorbit spaces \cite{feichtinger1989banach, velthoven2022quasi} is that of a so-called \emph{control weight}.
In the following lemma, we show the existence of such a weight for Peetre-type spaces and show that it can be estimated by standard envelopes as defined in Definition \ref{def:standard-env}. The construction of the control weight follows
  \cite[Lemma~5.7]{KvVV2021anisotropic}, but besides the slightly different
  parameters, the case distinction for the right translation needs to
  be accommodated with a few extra terms.
  The details are as follows.

\begin{lemma}\label{lem:ControlWeights}
  Let $\alpha \in \R$ and $\beta > 0$. For $q \in (0, \infty]$,
  set $r := \min\{1, q \}$.
  Then there exists a continuous, submultiplicative weight
  $w = w^{\alpha,\beta}_{\infty,q} : G_A \to [1,\infty)$ such that
  \begin{equation*}
    w(g) = \Delta^{1/r} (g^{-1}) \, w(g^{-1}),
     \qquad
     \vertiii{L_{g^{-1}}} \leq w(g),
    \qquad
 \vertiii{R_{g}}
    \leq w(g), \quad g \in G_A,
    \label{eq:ControlWeightProperties}
  \end{equation*}
  with implicit constant depending on $A,\beta$.
  The weight $w$ will be referred to as the \emph{standard control weight}.

  Moreover, for $q \in (0, \infty)$, set
  $\sigma_1 := |\det A|^{1/r + |\alpha-1/q|}$ and
  $\sigma_2 := |\det A|^{-|\alpha-1/q|}$, as well as
  \begin{equation*}
    \kappa_1
    := \begin{cases}
         |\det A|^{1/r+\alpha+\beta-1/q}, & \text{if } \alpha \geq -\frac{1/r+\beta-3/q}{2}, \\[0.1cm]
         |\det A|^{-(\alpha-2/q)},        & \text{otherwise} ,
       \end{cases}
    \end{equation*}
    and
    \begin{equation*}
    \kappa_2
    := \begin{cases}
         |\det A|^{-(\alpha+\beta-1/q)}, & \text{if } \alpha \geq -\frac{1/r+\beta-3/q}{2}, \\[0.1cm]
         |\det A|^{1/r + \alpha - 2/q},     & \text{otherwise} .
       \end{cases}
    \label{eq:KappaDefinition}
  \end{equation*}
  For $q = \infty$, set
  $\sigma_1 := |\det A|^{1 + |\alpha|}$ and
  $\sigma_2 := |\det A|^{-|\alpha|}$, as well as
  \begin{equation*}
    \kappa_1
    := \begin{cases}
         |\det A|^{1+\alpha+\beta}, & \text{if } \alpha \geq -\frac{\beta}{2}, \\[0.1cm]
         |\det A|^{-(\alpha-1)},        & \text{otherwise} ,
       \end{cases}
       \qquad \text{and} \qquad
       \kappa_2
    := \begin{cases}
         |\det A|^{-(\alpha+\beta)}, & \text{if } \alpha \geq -\frac{\beta}{2}, \\[0.1cm]
         |\det A|^{\alpha},     & \text{otherwise} .
       \end{cases}
    \end{equation*}
  Then the standard control weight $w$ satisfies
  \(
    w \asymp \Xi_{\sigma, 0} + \Xi_{\kappa, -\beta} .
  \)
\end{lemma}
\begin{proof}
 By \cite[Lemma~5.2]{KvVV2021anisotropic},  the weight $v : G_A \to [0,\infty)$ appearing in
  Lemma~\ref{lem:TranslationNormBounds} is submultiplicative, measurable, and
  locally bounded.  It also
  satisfies $v \geq 1$, and is therefore a weight function in the
  sense of \cite[Definition~3.7.1]{reiter2000}.  Hence, there exists a
  \emph{continuous}, submultiplicative function
  $v_0 : G_A \to [1,\infty)$ with $v \asymp v_0$ by the proof of
  \cite[Theorem~3.7.5]{reiter2000}.

  Let $\tau \in \R$ and define
  $a_{\tau}(g) = a_{\tau} (x,s) := |\det A|^{s \tau}$ for
  $g = (x,s) \in G_A$.  Then the weight function
  $w_{\gamma, \delta, \zeta} : G_A \to [1,\infty)$ with respect to
  $\gamma, \delta, \zeta \in \R$ is given by
  \begin{align*}
    w_{\gamma, \delta, \zeta} := \max \big\{ &1, \,\,\, a_{1/r}, \,\,\, a_{\gamma}, \,\,\,
    a_{-\gamma}, \,\,\, a_{\gamma + 1/r}, \,\,\, a_{1/r - \gamma}, \,\,\, \\
    &a_{\delta + 1/r} \cdot (v_0^\vee)^\beta, \,\,\, a_{-\delta} \cdot
    v_0^\beta, \,\,\, a_{\zeta + 1/r} \cdot (v_0^\vee)^\beta, \,\,\, a_{-\zeta} \cdot
    v_0^\beta \big\}.
  \end{align*}
  Clearly $w_{\gamma, \delta, \zeta}$ is again continuous and submultiplicative.
  As $\Delta = a_{-1}$ and $a_{\tau}^{\vee} = a_{-\tau}$, it is easily verified that
  \begin{align*}
    (\Delta^{1/r})^{\vee} \cdot w^{\vee}_{\gamma, \delta, \zeta}
    = w_{\gamma, \delta, \zeta} .
  \end{align*}
  We choose the parameters $\gamma, \delta, \zeta$ according to
  Lemma~\ref{lem:TranslationNormBounds}, i.e.,
  $\gamma:= \alpha - 1/q$, $\delta:= \alpha - 2/q$,
  $\zeta:= \alpha - 1/q$ if $q \in (0,\infty)$, and $\gamma:= \alpha$,
  $\delta:= \alpha - 1$, $\zeta:= \alpha$ otherwise. Set
  $w= w^{\alpha, \beta}_{\infty,q}:=w_{\gamma, \delta, \zeta}$.  For $g = (x,s) \in G_A$, an application of
  Lemma~\ref{lem:TranslationNormBounds} implies
  $\vertiii{L_{g^{-1}}} \lesssim a_{-\gamma}(g) \leq w(g)$ and
  \begin{equation*}
     \vertiii{R_{g}} \lesssim
    \begin{dcases}
      &a_{-\delta}(g)  v_0(g)^{\beta}, & \text{if } s > 0    \\[0.1cm]
      &a_{-\zeta}(g)  v_0(g)^{\beta}, & \text{if } s \leq 0
    \end{dcases}
    \leq w(g).
  \end{equation*}

  It remains to show that the control weight can be estimated by
  standard envelopes.  To this end, note that $w \asymp w_1 + w_2$ for
  the weights given by
  ${w_1 := \max \{ a_0, a_{1/r}, a_\gamma, a_{-\gamma},
    a_{1/r+\gamma}, a_{1/r-\gamma} \}}$ and
  \( w_2 := \max \bigl\{ a_{\delta + 1/r} \cdot (v_0^{\vee})^\beta ,
  \,\, a_{-\delta} \cdot v_0^\beta, \,\, a_{\zeta + 1/r} \cdot
  (v_0^{\vee})^\beta , \,\, a_{-\zeta} \cdot v_0^\beta \bigr\} \).  A
  straightforward computation, also done in the proof of
  \cite[Lemma~5.7]{KvVV2021anisotropic}, reveals that
  \begin{equation*}
    w_1(g) = w_1(x,s)
     \leq \begin{dcases}
      &|\det A|^{s  (1/r + |\gamma|)}, & \text{if } s \geq 0, \\
      &|\det A|^{- s |\gamma|}, & \text{if } s <    0, \\
    \end{dcases}
    = \theta_\sigma (s) = \Xi_{\sigma,0}(x,s),
  \end{equation*}
  by the choice of $\sigma$.
  For estimating $w_2$, we use the fact that
  $v_0(x,s) \asymp |\det A|^{\NegPart{s}} \, \eta_{-1} (x,s)$ and
  $v_0^\vee(x,s) \asymp |\det A|^{\PosPart{s}} \, \eta_{-1} (x,s)$ as
  shown in \cite[Lemma~5.7]{KvVV2021anisotropic}. For $s \geq 0$, this gives
  \begin{align*}
    w_2(x,s)
    & \asymp \big( \eta_{-1}(x,s) \big)^{\beta}
      \max \big\{
      |\det A|^{(1/r + \delta + \beta) s} , \,\,
      |\det A|^{-\delta s}, \,\,
      |\det A|^{(1/r + \zeta + \beta) s} , \,\,
      |\det A|^{-\zeta s}
      \big\} \\
   & = \eta_{-\beta}(x,s)  \kappa_1^s
      = \Xi_{\kappa,-\beta} (x,s) ,
  \end{align*}
  since, in the case of $q \in (0, \infty)$, we have
  \[ \max \{ -\delta, \,\,\, 1/r + \zeta + \beta \} = \max \{ -\alpha +
  2/q, \,\,\, 1/r + \alpha + \beta - 1/q, \} = 1/r + \alpha + \beta -
  1/q \] if and only if $\alpha \geq -\frac{1/r+\beta-3/q}{2}$,
  with a similar case distinction for $q = \infty$.
  The estimate for $s < 0$ follows analogously.
\end{proof}

\subsection{Coorbit spaces}
\label{sub:CoorbitSpaces}
The aim of this section is to show that Triebel-Lizorkin spaces $\TLi$ can be identified with so-called \emph{coorbit spaces} \cite{feichtinger1989banach, velthoven2022quasi} by use of Theorem \ref{thm:norm_equiv}.

The definition of coorbit spaces requires the notion of a local maximal function. For a function $F \in L^{\infty}_{\loc} (G_A)$, its (left-sided) \emph{maximal function} is defined by
\begin{align} \label{eq:left_maximal}
M_Q^L F (g) = \esssup_{u \in Q} |F(g u)|, \quad g \in G_A,
\end{align}
where $Q \subset G_A$ is a relatively compact unit neighborhood.

\begin{definition}\label{def:PeetreCoorbitDefinition}
  Let $q \in (0,\infty]$, $\alpha \in \R$, and $\beta > 0$.
  Let $A \in \mathrm{GL}(d, \mathbb{R})$ be expansive and  exponential,
   and let $Q \subset G_A$ be a relatively compact, symmetric unit neighborhood.

  For admissible $\psi \in \SC_0(\mathbb{R}^d)$,
  the \emph{coorbit space} $\Co (\PTi) = \Co_{\psi} (\PTi)$  is the collection of all $f \in \SC_0'(\R^d)$ satisfying
  \[
    \| f \|_{\Co(\PTi)} = \| f \|_{\Co_{\psi} (\PTi)}
    = \big\| M^L_Q (W_{\psi} f) \big\|_{\PTi}
    < \infty.
  \]
  The space will be equipped with the (quasi-)norm $\| \, . \, \|_{\Co(\PT)}$.
\end{definition}

The space $\Co(\PTi)$ as defined in \Cref{def:PeetreCoorbitDefinition} is complete with respect to the quasi-norm $\| \cdot \|_{\Co(\TLi)}$.
In addition, its definition is independent of the chosen defining vector $\psi$ and unit neighborhood $Q$,
with equivalent norms for different choices.
These basic properties follow from the general theory
\cite{velthoven2022quasi}; see \cite[Remark 5.10]{KvVV2021anisotropic} for details and references.

The following is the key result of this section.

\begin{proposition}\label{prop:TL_coorbit}
  Let $\alpha \in \mathbb{R}$ and $q \in (0,\infty]$. Let $\beta > 1/q$ if $q \in (0, \infty)$,
  and $\beta > 1$ if $q = \infty$.
  Then
  \begin{equation*}
    \TLi =
    \begin{cases}
      \Co \bigl(\PTalti{-(\alpha+1/2-1/q)}\bigr),
      &q \in (0, \infty), \\
      \Co \bigl(\PTaltii{-(\alpha+1/2)}\bigr),
      &q = \infty.
    \end{cases}
  \end{equation*}
\end{proposition}

\begin{proof}

  Throughout, let $\psi \in \SC(\R^d)$ be an admissible vector
  with $\widehat{\psi} \in C^{\infty}_c (\mathbb{R}^d \setminus \{0\}$ satisfying the support condition
  \eqref{eq:analyzing_positive}.  The existence of such vectors follows by
  \cite[Proposition 10]{currey2016integrable}; see also Lemma~\ref{lem:admissible_expansive}
  Furthermore, define $Q := [-1, 1)^d \times [-1,1)$.

  We split the proof into three steps and consider only
  $q \in (0, \infty)$. The case $q = \infty$ follows with the usual
  adjustments. Set $\alpha':= \alpha + 1/2 - 1/q$. The first two steps are essentially identical to the first part of the proof of \cite[Proposition 5.11]{KvVV2021anisotropic},
  but included for completeness.
  \\\\
  \textbf{Step 1.} Note that $\psi^{\ast} \in \SC_0(\R^d)$ also satisfies
  \eqref{eq:analyzing_positive}, where $\psi^{\ast} (t) = \overline{\psi(-t)}$.  A direct calculation gives
  \begin{align} \label{eq:wavelet_conv}
    W_{\psi}f(x,s) = \langle f, \pi(x,s) \psi \rangle
    = \langle f , |\det A|^{s/2} T_x \psi_{-s} \rangle
    = |\det A|^{s/2} f \ast \psi^{\ast}_{-s}(x) .
  \end{align}
  By \cite[Lemma A.1]{KvVV2021anisotropic}, it holds for arbitrary $\beta > 0$ and $s \in \mathbb{R}$ that
  \[
  (\psi^{\ast})_{s, \beta}^{**} f(x): = \sup_{z \in \mathbb{R}^d} \frac{|(f \ast \psi^{\ast}_s) (x+z)|}{(1+\rho_A (A^s z))^{\beta}} = \esssup_{z \in \mathbb{R}^d} \frac{|(f \ast \psi^{\ast}_s) (x+z)|}{(1+\rho_A (A^s z))^{\beta}}, \quad x \in \mathbb{R}^d.
  \]
  Therefore, an application of Theorem~\ref{thm:norm_equiv} yields
  \begin{align*}
    \|f\|_{\TLi}^q
    &\asymp  \sup_{\ell \in \Z, w \in \R^d}  \bigg(
      \dashint_{Q_{\ell} (w)} \int_{-\ell}^\infty \bigg[ |\det A|^{\alpha s}
      \esssup_{z \in \mathbb{R}^d} \frac{| (f \ast \psi^{\ast}_{s})(x + z)|} {(1 +
      \rho_A (A^{s} z))^{\beta}} \bigg]^q ds \,
      dx \bigg)\\
    &=  \sup_{\ell \in \Z, w \in \R^d}  \bigg(
      \dashint_{Q^{\ell} (w)} \int_{-\infty}^\ell \bigg[ |\det A|^{- \alpha s}
      \esssup_{z \in \mathbb{R}^d} \frac{| (f \ast \psi^{\ast}_{-s})(x + z)|} {(1 +
      \rho_A (A^{-s} z))^{\beta}} \bigg]^q ds \,
      dx \bigg)\\
    &=  \sup_{\ell \in \Z, w \in \R^d}  \bigg(
      \dashint_{Q_{\ell} (w)} \int_{-\infty}^\ell \bigg[ |\det A|^{- (\alpha +1/2 -1/q) s}
      \esssup_{z \in \mathbb{R}^d} \frac{|W_\psi f (x+z,s)|} {(1 +
      \rho_A (A^{-s} z))^{\beta}} \bigg]^q \frac{ds \, dx}{|\det A|^{s}}
       \bigg)
      \\
    &= \|W_\psi f\|_{\PTalti{-\alpha'}}^q \numberthis \label{eq:TL_normequiv}
  \end{align*}
  for any $f \in \SC_0'(\R^d)$.
\\\\
\textbf{Step 2.} The estimate $|W_{\psi} f| \leq M^L_Q W_{\psi} f$ a.e.
implies that
\[
\| f \|_{\TLi} \asymp \|W_\psi f\|_{\PTalti{-\alpha'}} \leq \| M^L_Q (W_{\psi} f) \|_{\PTalti{-\alpha'}} = \| f \|_{\Co (\PTalti{-\alpha'})},
\]
for $f \in \SC_0'(\R^d)$.
\\\\
\textbf{Step 3.} We prove the remaining estimate
$\| f \|_{\Co (\PTalt{-\alpha'})} \lesssim \| f \|_{\TL}$ for
$f \in \SC_0'(\R^d)$.  In
 \cite[Equation~(5.12)]{KvVV2021anisotropic}, we already showed that
\begin{equation*}
  \bigg(|\det A|^{-(\alpha+1/2)s} \esssup_{\substack {z \in \mathbb{R}^d}}
  \frac{ | M^L_Q(W_{\psi} f) (x + z, s)|}
  {(1 + \rho_{A}(A^{-s} z))^{\beta}}  \bigg)^q
  \lesssim \sum_{k = - N}^N \bigg ( |\det A|^{- \alpha (s+k)} (\psi^{\ast})^{**}_{-(s+k), \beta} f (x)
  \bigg)^q,
\end{equation*}
with implicit constant depending on $A, \alpha, \beta,q$ and $N$, where $N \in \N$ depends on the support of $\psi$.
Combining this with Theorem~\ref{thm:norm_equiv} yields
\begin{align*}
  &\| f \|_{\Co(\PTalt{-\alpha'})}^q \\
  & = \sup_{\ell \in \Z, w \in \R^d}  \bigg(
    \dashint_{Q_{\ell}(w)} \int_{-\infty}^\ell \bigg[ |\det A|^{-(\alpha+1/2-1/q) s}
    \esssup_{z \in \mathbb{R}^d} \frac{| M^L_Q(W_{\psi} f)(x+z,s)|} {(1 +
    \rho_A (A^{-s} z))^{\beta}} \bigg]^q \frac{ds}{|\det A|^{s}} \,
    dx \bigg) \\
  & \lesssim \sup_{\ell \in \Z, w \in \R^d}  \bigg(
    \dashint_{Q_{\ell}(w)} \int_{-\infty}^\ell
    \sum_{k = - N}^N \bigg ( |\det A|^{- \alpha (s+k)} (\psi^{\ast})^{**}_{-(s+k), \beta} f (x)
    \bigg)^q ds \,
    dx \bigg) \\
  & \lesssim_N \sup_{\ell \in \Z, w \in \R^d}  \bigg(
    \dashint_{Q_{\ell}(w)} \int_{-\infty}^{\ell+N}
    \bigg ( |\det A|^{- \alpha s} (\psi^{\ast})^{**}_{-s, \beta} f (x)
    \bigg)^q ds \,
    dx \bigg)  \\
  & = \sup_{\ell \in \Z, w \in \R^d}   \bigg(
    \frac{\Lebesgue{A^{\ell+N} \Omega}}{\Lebesgue{A^{\ell} \Omega}}
    \dashint_{Q_{\ell+N}(w)} \int_{-(\ell+N)}^{\infty}
    \bigg ( |\det A|^{ \alpha s} (\psi^{\ast})^{**}_{s, \beta} f (x)
    \bigg)^q ds \,
    dx \bigg) \\
  & \lesssim_N \| f \|_{\TLi}^q,
\end{align*}
where we used the fact that
$ \frac{\Lebesgue{A^{\ell+N} \Omega}}{\Lebesgue{A^{\ell} \Omega}} =
|\det A|^N \lesssim_N 1$ in the last step.
\end{proof}

\section{Molecular decompositions} \label{sec:molecular}
In this section we use the identification of Triebel-Lizorkin spaces and associated coorbit spaces (cf. \Cref{prop:TL_coorbit}) to obtain proofs of
\Cref{thm:main2_intro} and \Cref{thm:main3_intro}.

\subsection{Molecular systems}
In addition to the left-sided local maximal function defined in Equation \eqref{eq:left_maximal}, we will also need a two-sided version, defined by
\[
M_Q F (g) = \esssup_{u,v \in Q} |F(u g v)|, \quad g \in G_A,
\]
for $F \in L^{\infty}_{\loc} (G_A)$, with $Q \subset G_A$ being a fixed relatively compact unit neighborhood.

For $r = \min \{1,q\}$ with $q \in (0,\infty]$ and the standard control weight $w = w_{\infty, q}^{\alpha, \beta} : G_A \to [1,\infty)$ provided by \Cref{lem:ControlWeights}, define the associated \emph{Wiener amalgam space} $\WLwr$ by
\[
\WLwr = \bigg\{ F \in L^{\infty}_{\loc} (G_A) : M_Q F \in L^r_w (G_A) \bigg\}.
\]
The space $\WLwr$ is independent of the choice of neighborhood $Q$ and is complete with respect to the quasi-norm $\| F \|_{\WLwr} := \| M_Q F\|_{L^r_w}$; see, e.g., \cite[Section 2]{rauhut2007wiener} and \cite[Section 2]{velthoven2022quasi}.

The space $\WLwr$ provides the class of envelopes that will be used for defining molecules.

\begin{definition} \label{def:molecule}
Let $\psi \in L^2 (\mathbb{R}^d)$ be a non-zero vector such that $W_{\psi} \psi \in \WLwr$ for the standard control weight $w : G_A \to [1,\infty)$ defined in \Cref{lem:ControlWeights}.

A countable family $(\phi_{\gamma})_{\gamma \in \Gamma}$ in $L^2 (\mathbb{R}^d)$ is called an \emph{$L^r_w$-molecular system} if there exists $\Phi \in \WLwr$ such that
\begin{align} \label{eq:molecule_def}
|W_{\psi} \phi_{\gamma} (x) | \leq \Phi (\gamma^{-1} x), \quad \gamma \in \Gamma, x \in G.
\end{align}
A function $\Phi \in \WLwr$ satisfying \eqref{eq:molecule_def} is called an \emph{envelope} for $(\phi_{\gamma})_{\gamma \in \Gamma}$.
\end{definition}

The following lemma provides an extended duality pairing for $\mathcal{S}_0' (\R^d)$. See \cite[Lemma 6.8]{KvVV2021anisotropic} for its proof.

\begin{lemma}[\cite{KvVV2021anisotropic}]
  \label{lem:ExtendedDualPairing}
  Let $\psi \in \mathcal{S}_0 (\mathbb{R}^d)$ be an admissible vector
  and let $w : G_A \to [1,\infty)$ be the standard control weight as
  defined in \Cref{lem:ControlWeights}.

  Suppose $f \in \mathcal{S}_0' (\mathbb{R}^d)$ satisfies
  $W_{\psi} f \in L^{\infty}_{1/w} (G_A)$ and
  $\phi \in L^2 (\mathbb{R}^d)$ satisfies
  $W_{\psi} \phi \in L^{1}_w (G_A)$.
  Then the \emph{extended pairing} defined by
  \[
    \langle f, \phi \rangle_{\psi} := \int_{G_A} W_{\psi} f (g)
    \overline{W_{\psi} \phi (g)} \; d\mu_{G_A} (g) \in \mathbb{C}
  \]
  is well-defined and independent of the choice of admissible vector
  $\psi \in \mathcal{S}_0 (\mathbb{R}^d)$.
\end{lemma}

\subsection{Sequence spaces}
The following definition provides a class of sequence spaces that will be used in the molecular decomposition of anisotropic Triebel-Lizorkin spaces.

\begin{definition} \label{def:peetre_seq}
Let $U \subseteq G_A$ be a relatively compact unit neighborhood with non-void interior and let $\Gamma$ be any family in  $G_A$ such that $\sup_{g \in G_A} \# (\Gamma \cap g U ) < \infty$.

For $\alpha \in \mathbb{R}$, $\beta > 0$ and $q \in (0,\infty]$,
the \emph{Peetre-type sequence space} $\PTseq (\Gamma,  U)$ consists of all sequences $c \in \mathbb{C}^{\Gamma}$ satisfying
\[
    \| c \|_{\PTseq}
    := \bigg\|
         \sum_{\gamma \in \Gamma}
           |c_{\gamma}| \Indicator_{\gamma U}
       \bigg\|_{\PTi}
    < \infty,
  \]
  where $\PTi (G_A)$ denotes the Peetre-type space defined in Definition \ref{def:PTi}.
\end{definition}

The Peetre-type sequence space $\PTseq (\Gamma, U)$ forms a quasi-Banach space with respect to the quasi-norm $\| \cdot \|_{\PTseq}$. In addition, it is
independent of the choice of the defining neighborhood $U$.
For proofs of both facts, see, e.g.,~\cite{rauhut2007wiener,feichtinger1989banach}
or \cite[Lemma~2.3.16]{VoigtlaenderPhDThesis}.

\subsection{Proofs of \Cref{thm:main2_intro} and \Cref{thm:main3_intro}}

\Cref{thm:main2_intro} and \Cref{thm:main3_intro} will be obtained from corresponding results for abstract coorbit spaces \cite{velthoven2022quasi}.

\begin{proof}[Proof of \Cref{thm:main2_intro}]
  For the fact that if $W_{\varphi} \psi \in \WLwr$ for some admissible $\varphi \in \mathcal{S}_0 (\mathbb{R}^d)$, then $W_{\phi} \psi \in \WLwr$ for all $\phi \in \mathcal{S}_0 (\mathbb{R}^d)$, cf. the proof of \cite[Lemma 6.9]{KvVV2021anisotropic}.
  Throughout, we fix an admissible $\varphi \in \SC_0 (\R^d)$ with
  $\widehat{\varphi} \in C^\infty_c(\mathbb{R}^d \setminus \{0\})$ satisfying condition
  \eqref{eq:analyzing_positive}.  Such an admissible vector exists by
  Lemma~\ref{lem:admissible_expansive}; see \cite[Proposition
  10]{currey2016integrable}.
  \\\\
  \textbf{Step 1.}
  Under the assumptions, it follows by an application of Proposition~\ref{prop:TL_coorbit}
  that   $\TLi = \Co_{\varphi} (\PTalti{-\alpha'})$.
  For applying the relevant results of \cite{velthoven2022quasi}, it will next be shown that
  $\Co_{\varphi} (\PTalti{-\alpha'})$ can be identified with the abstract
  coorbit spaces used in \cite{velthoven2022quasi}.  To this end, note
  that $\PTalti{-\alpha'}$ is a solid, translation invariant
  quasi-Banach space by Lemma~\ref{lem:TranslationNormBounds}, which
  satisfies the $r$-norm property by Lemma~\ref{lem:r-norm-property}.
  The standard control weight
  $w = w^{-\alpha', \beta}_{\infty,q}: G_A \to [1, \infty)$
  of Lemma~\ref{lem:ControlWeights} is a
  \emph{strong control weight} in the sense of \cite[Definition
  3.1]{velthoven2022quasi}, which can be dominated by a linear combination of
  standard envelopes.

  Define the space
  $\mathcal{H}^1_{w} (\psi) := \big\{ f \in L^2
  (\mathbb{R}^d) : W_{\psi} f \in L^1_{w} (G_A) \big\}$ and
  equip it with the norm
  $\| f \|_{\mathcal{H}^1_{w} (\psi)} := \| W_{\psi} f
  \|_{L^1_{w}}$. Set
  $\mathcal{R}_{w} (\psi) := ( \mathcal{H}^1_{w}
  (\psi) )^*$ to be its anti-dual space.  Since both
  $W_\psi \psi, W_\varphi \psi \in \WLwr$, it follows by similar arguments as proving
  \cite[Lemma~D.1]{KvVV2021anisotropic} that
  \begin{equation*}
    \Co\nolimits_{\psi}^{\mathcal{H}}(\PTalti{-\alpha'})
    := \big\{ f  \in \mathcal{R}_{w} (\psi) :
    M^L_Q V_{\psi} f \in \PTalti{-\alpha'} \big\}
    = \Co\nolimits_{\varphi} (\PTalti{-\alpha'}),
  \end{equation*}
  where $V_{\psi} f = \langle f, \pi (\cdot) \psi \rangle_{\mathcal{R}_w, \mathcal{H}^1_w}$ denotes the conjugate-linear pairing.
  In addition, for the unique extension
  $\widetilde{f} \in \Co_{\psi}^{\mathcal{H}}
  (\PTalti{-\alpha'})$ of
  $f \in \Co_{\psi} (\PTalti{-\alpha'})$, it holds that
  \begin{equation}
    \label{eq:dual-pairing-identification}
    \langle \widetilde{f}, \phi \rangle_{\mathcal{R}_{w},
      \mathcal{H}^1_{w}} = \langle f, \phi \rangle_{\varphi}, \qquad \phi
    \in \mathcal{H}^1_{w} (\varphi),
  \end{equation}
  where $\langle \cdot , \cdot \rangle_{\varphi}$ denotes the extended
  pairing of \Cref{lem:ExtendedDualPairing}.
\\\\
\textbf{Step 2.}
  Applying \cite[Theorem~6.14]{velthoven2022quasi} yields a compact
  unit neighborhood $U \subset G_A$ such that for any set $\Gamma \subset G_A$ satisfying \eqref{eq:well-spread}, there exists a family
  $(\phi_\gamma)_{\gamma \in \Gamma}$ of $L^r_w$-localized molecules
  in $L^2(\R^d)$ such that any
  $\widetilde{f} \in \Co_{\psi}^{\mathcal{H}}
  (\PTalti{-\alpha'})$ can be represented as
  \begin{equation}
    \widetilde{f}
    = \sum_{\gamma \in \Gamma}
    \langle \widetilde{f}, \pi(\gamma) \psi \rangle_{\mathcal{R}_{w},
      \mathcal{H}^1_{w}} \, \phi_\gamma
    = \sum_{\gamma \in \Gamma}
    \langle \widetilde{f}, \phi_\gamma \rangle_{\mathcal{R}_{w},
      \mathcal{H}^1_{w}}  \, \pi(\gamma) \psi,
    \label{eq:FrameProofTargetIdentity}
  \end{equation}
  with unconditional convergence with respect to the weak$^*$-topology
  on $\mathcal{R}_{w}(\psi)$.

  Given any $f \in \TLi = \Co_{\varphi} (\PTalti{-\alpha'})$,
  we apply \eqref{eq:FrameProofTargetIdentity} to its unique
  extension $\widetilde{f} \in \Co_{\psi}^{\mathcal{H}}
  (\PTalti{-\alpha'})$.  Since
  $\WLwr \hookrightarrow L^1_{w}$ by
  \cite[Lemma~3.3(ii)]{velthoven2022quasi}, we have that
  $\psi \in \mathcal{H}^1_{w}$ and that
  $\phi_\gamma \in \mathcal{H}^1_{w}$.  Consequently, the dual pairings in
  \eqref{eq:FrameProofTargetIdentity} agree with the extended pairing
  of \Cref{lem:ExtendedDualPairing}; see Equation
  \eqref{eq:dual-pairing-identification}. Since $\SC_0 \hookrightarrow \mathcal{H}_{w}^1$,
   restricting both sides of \eqref{eq:FrameProofTargetIdentity}
  to $\SC_0$ yields the desired expansion with unconditional
  convergence in the weak$^*$-topology of $\SC_0' \cong \SP$.
\end{proof}

\begin{proof}[Proof of \Cref{thm:main3_intro}]
  As shown in Step 1 of the proof of \Cref{thm:main2_intro}, the map $f \mapsto f|_{\mathcal{S}_0}$ is a well-defined bijection from
  $ \Co_{\psi}^{\mathcal{H}}(\PTalti{-\alpha'})$ into $\TLi$,
  with $\Co^{\mathcal{H}}_{\psi} ( \PTalti{-\alpha'})$ denoting the abstract coorbit space as defined in \cite{velthoven2022quasi}. Therefore, applying \cite[Theorem~6.15]{velthoven2022quasi} yields a compact
  unit neighborhood $U \subset G_A$ such that for any
  set $\Gamma \subset G_A$ satisfying \eqref{eq:separated}, there exists a family
  $(\phi_\gamma)_{\gamma \in \Gamma}$ in
  $ \overline{\Span} \{ \pi(\gamma) \psi : \gamma \in \Gamma \}
  \subset L^2(\R^d)$ of $L^r_w$-localized molecules such that the
  moment problem
  \begin{equation*}
    \langle \widetilde{f}, \phi(\gamma)\psi \rangle_{\mathcal{R}_{w},
      \mathcal{H}^1_{w}} = c_\gamma, \qquad \gamma \in \Gamma,
  \end{equation*}
  admits the solution
  $\widetilde{f} := \sum_{\gamma \in \Gamma} c_\gamma \phi_\gamma \in
  \Co_{\psi}^{\mathcal{H}}(\PTalti{-\alpha'})$ for any
  sequence
  $(c_\gamma)_{\gamma \in \Gamma} \in
  \PTseq[-\alpha'](\Gamma)$.  Furthermore, by
  arguing as in the proof of \Cref{thm:main2_intro}, we can restrict
  to
  \[ f:= \widetilde{f}|_{\mathcal{S}_0} \in \Co\nolimits_{\varphi}
  (\PTalti{-\alpha'}) = \TLi \] to obtain the desired claim.
\end{proof}

\section*{Acknowledgements}
S.~K. ~was supported by project I 3403 of the Austrian Science Fund (FWF).
J.v.V. gratefully acknowledges support from the Austrian Science Fund (FWF) project J-4445.
F.\ Voigtlaender acknowledges support by the German Research Foundation (DFG)
in the context of the Emmy Noether junior research group VO 2594/1--1.


\begin{thebibliography}{10}

\bibitem{bownik2003anisotropic}
M.~Bownik.
\newblock Anisotropic {H}ardy spaces and wavelets.
\newblock {\em Mem. Amer. Math. Soc.}, 164(781):vi+122, 2003.

\bibitem{bownik2005atomic}
M.~Bownik.
\newblock Atomic and molecular decompositions of anisotropic {B}esov spaces.
\newblock {\em Math. Z.}, 250(3):539–571, 2005.

\bibitem{bownik2007anisotropic}
M.~Bownik.
\newblock Anisotropic {T}riebel-{L}izorkin spaces with doubling measures.
\newblock {\em J. Geom. Anal.}, 17(3):387–424, 2007.

\bibitem{bownik2008duality}
M.~Bownik.
\newblock Duality and interpolation of anisotropic {T}riebel-{L}izorkin spaces.
\newblock {\em Math. Z.}, 259(1):131–169, 2008.

\bibitem{bownik2006atomic}
M.~Bownik and K.-P. Ho.
\newblock Atomic and molecular decompositions of anisotropic
  {T}riebel-{L}izorkin spaces.
\newblock {\em Trans. Amer. Math. Soc.}, 358(4):1469--1510, 2006.

\bibitem{bui2000characterization}
H.-Q. {Bui} and M.~H. {Taibleson}.
\newblock {The characterization of the Triebel-Lizorkin spaces for
  \(p=\infty\)}.
\newblock {\em {J. Fourier Anal. Appl.}}, 6(5):537--550, 2000.

\bibitem{CoifmanSpacesOfHomogeneousType}
R.~R. Coifman and G.~Weiss.
\newblock {\em Analyse harmonique non-commutative sur certains espaces
  homog\`enes}.
\newblock Lecture Notes in Mathematics, Vol. 242. Springer-Verlag, Berlin-New
  York, 1971.

\bibitem{coifman1977extensions}
R.~R. Coifman and G.~Weiss.
\newblock Extensions of {H}ardy spaces and their use in analysis.
\newblock {\em Bull. Amer. Math. Soc.}, 83(4):569--645, 1977.

\bibitem{currey2016integrable}
B.~Currey, H.~F\"uhr, and K.~Taylor.
\newblock Integrable wavelet transforms with abelian dilation groups.
\newblock {\em J. Lie Theory}, 26(2):567--596, 2016.

\bibitem{feichtinger1989banach}
H.~G. Feichtinger and K.~H. Gr\"{o}chenig.
\newblock Banach spaces related to integrable group representations and their
  atomic decompositions. {I}.
\newblock {\em J. Funct. Anal.}, 86(2):307–340, 1989.

\bibitem{frazier1990discrete}
M.~Frazier and B.~Jawerth.
\newblock A discrete transform and decompositions of distribution spaces.
\newblock {\em J. Funct. Anal.}, 93(1):34–170, 1990.

\bibitem{fuehr2002continuous}
H.~F\"{u}hr and M.~Mayer.
\newblock Continuous wavelet transforms from semidirect products: cyclic
  representations and {P}lancherel measure.
\newblock {\em J. Fourier Anal. Appl.}, 8(4):375–397, 2002.

\bibitem{gilbert2002smooth}
J.~E. Gilbert, Y.~S. Han, J.~A. Hogan, J.~D. Lakey, D.~Weiland, and G.~Weiss.
\newblock Smooth molecular decompositions of functions and singular integral
  operators.
\newblock {\em Mem. Amer. Math. Soc.}, 156(742):viii+74, 2002.

\bibitem{grafakos2014modern}
L.~Grafakos.
\newblock {\em Modern {F}ourier analysis}, volume 250 of {\em Graduate Texts in
  Mathematics}.
\newblock Springer, New York, third edition, 2014.

\bibitem{grafakos2009vector}
L.~Grafakos, L.~Liu, and D.~Yang.
\newblock Vector-valued singular integrals and maximal functions on spaces of
  homogeneous type.
\newblock {\em Math. Scand.}, 104(2):296–310, 2009.

\bibitem{groechenig1991describing}
K.~Gr\"{o}chenig.
\newblock Describing functions: atomic decompositions versus frames.
\newblock {\em Monatsh. Math.}, 112(1):1–42, 1991.

\bibitem{groechenig1992compact}
K.~Gr\"{o}chenig, E.~Kaniuth, and K.~F. Taylor.
\newblock Compact open sets in duals and projections in {$L^1$}-algebras of
  certain semi-direct product groups.
\newblock {\em Math. Proc. Cambridge Philos. Soc.}, 111(3):545–556, 1992.

\bibitem{grochenig2009molecules}
K.~{Gr\"ochenig} and M.~{Piotrowski}.
\newblock {Molecules in coorbit spaces and boundedness of operators}.
\newblock {\em {Stud. Math.}}, 192(1):61--77, 2009.

\bibitem{ho2003frames}
K.-P. Ho.
\newblock Frames associated with expansive matrix dilations.
\newblock {\em Collect. Math.}, 54(3):217–254, 2003.

\bibitem{KvVV2021anisotropic}
S.~Koppensteiner, J.~T. van Velthoven, and F.~Voigtlaender.
\newblock Anisotropic {T}riebel-{L}izorkin spaces and wavelet coefficient decay
  over one-parameter dilation groups, {I}.
\newblock {\em Preprint. arXiv: 2104.14361}, 2021.

\bibitem{laugesen2002characterization}
R.~S. Laugesen, N.~Weaver, G.~L. Weiss, and E.~N. Wilson.
\newblock A characterization of the higher dimensional groups associated with
  continuous wavelets.
\newblock {\em J. Geom. Anal.}, 12(1):89–102, 2002.

\bibitem{liang2012new}
Y.~{Liang}, Y.~{Sawano}, T.~{Ullrich}, D.~{Yang}, and W.~{Yuan}.
\newblock {New characterizations of Besov-Triebel-Lizorkin-Hausdorff spaces
  including coorbits and wavelets}.
\newblock {\em {J. Fourier Anal. Appl.}}, 18(5):1067--1111, 2012.

\bibitem{liu2019littlewood1}
J.~Liu, D.~Yang, and W.~Yuan.
\newblock Littlewood-{P}aley characterizations of weighted anisotropic
  {T}riebel-{L}izorkin spaces via averages on balls {I}.
\newblock {\em Z. Anal. Anwend.}, 38(4):397--418, 2019.

\bibitem{rauhut2007wiener}
H.~Rauhut.
\newblock Wiener amalgam spaces with respect to quasi-{B}anach spaces.
\newblock {\em Colloq. Math.}, 109(2):345–362, 2007.

\bibitem{reiter2000}
H.~Reiter and J.~D. Stegeman.
\newblock {\em Classical harmonic analysis and locally compact groups},
  volume~22 of {\em London Mathematical Society Monographs. New Series}.
\newblock The Clarendon Press, Oxford University Press, New York, second
  edition, 2000.

\bibitem{romero2020dual}
J.~L. {Romero}, J.~T. {van Velthoven}, and F.~{Voigtlaender}.
\newblock {On dual molecules and convolution-dominated operators}.
\newblock {\em {J. Funct. Anal.}}, 280(10):57, 2021.
\newblock Id/No 108963.

\bibitem{rychkov1991on}
V.~S. Rychkov.
\newblock On a theorem of {B}ui, {P}aluszy\'{n}ski, and {T}aibleson.
\newblock {\em Tr. Mat. Inst. Steklova}, 227(18):286--298, 1999.

\bibitem{schulz2004projections}
E.~Schulz and K.~F. Taylor.
\newblock Projections in {$L^1$}-algebras and tight frames.
\newblock In {\em Banach algebras and their applications}, volume 363 of {\em
  Contemp. Math.}, pages 313--319. Amer. Math. Soc., Providence, RI, 2004.

\bibitem{stein1970singular}
E.~M. {Stein}.
\newblock {\em {Singular integrals and differentiability properties of
  functions}}, volume~30.
\newblock Princeton University Press, Princeton, NJ, 1970.

\bibitem{stroemberg1989weighted}
J.-O. Str\"{o}mberg and A.~Torchinsky.
\newblock {\em Weighted {H}ardy spaces}, volume 1381 of {\em Lecture Notes in
  Mathematics}.
\newblock Springer-Verlag, Berlin, 1989.

\bibitem{ullrich2012continuous}
T.~Ullrich.
\newblock Continuous characterizations of {B}esov-{L}izorkin-{T}riebel spaces
  and new interpretations as coorbits.
\newblock {\em J. Funct. Spaces Appl.}, 2012.
\newblock Art. ID 163213, 47.

\bibitem{velthoven2022quasi}
J.~T. {Van Velthoven} and F.~{Voigtlaender}.
\newblock Coorbit spaces and dual molecules: The quasi-{B}anach case.
\newblock {\em Preprint. arXiv:2203.07959}, 2022.

\bibitem{VoigtlaenderPhDThesis}
F.~Voigtlaender.
\newblock {\em Embedding Theorems for Decomposition Spaces with Applications to
  Wavelet Coorbit Spaces}.
\newblock PhD thesis, RWTH Aachen University, 2015.
\newblock \url{http://publications.rwth-aachen.de/record/564979}.

\bibitem{zaanen1967integration}
A.~C. Zaanen.
\newblock {\em Integration}.
\newblock North-Holland Publishing Co., Amsterdam; Interscience Publishers John
  Wiley \& Sons, Inc., New York, 1967.

\end{thebibliography}
\end{document}